\documentclass{amsart}
\usepackage{pgfplots}
\pgfplotsset{compat=1.15}

\usepackage{mathrsfs}
\usetikzlibrary{arrows}

\RequirePackage{amspaperstyle}
\usepackage{bbm}

\newcommand{\RR}{\mathbb{R}}
\newcommand{\NN}{\mathbb{N}}
\newcommand{\diag}{\operatorname{diag}}
\newcommand{\ZZ}{\mathbb{Z}}

\newcommand{\EE}{\mathbb{E}}
\newcommand{\PP}{\mathbb{P}}
\newcommand{\cF}{\mathcal{F}}
\newcommand{\cL}{\mathcal{L}}

\newcommand{\rot}{\mathrm{z}}
\newcommand{\disp}{\mathrm{disp}}

\begin{document}

\title{Birkhoff generic points on curves in horospheres}

\author{Omri Solan}
\address{Einstein Institute of Mathematics, Edmond J.~Safra Campus, Givat Ram, Hebrew University, Jerusalem, Israel}
\email{omrinisan.solan@mail.huji.ac.il}
\author{Andreas Wieser}
\address{Einstein Institute of Mathematics, Edmond J.~Safra Campus, Givat Ram, Hebrew University, Jerusalem, Israel}
\email{andreas.wieser@mail.huji.ac.il}
\thanks{
  Both authors are supported by the ERC grant HomDyn, ID 833423.
  The first named author is funded also by the generous donation of Dr.~Arthur A.~Kaselemas.}
\date{\today}

\begin{abstract}
  Let $\{a_t: t \in \R\}< \SL_d(\R)$ be a diagonalizable subgroup whose expanding horospherical subgroup $U < \SL_d(\R)$ is abelian.
  By the Birkhoff ergodic theorem, for any $x \in \SL_d(\R)/\SL_d(\Z)$ and for almost every point $u \in U$ the point $ux$ is Birkhoff generic for $a_t$ when $t \to \infty$.
  We prove that the same is true when $U$ is replaced by any non-degenerate analytic curve in $U$.

  This Birkhoff genericity result has various applications in Diophantine approximation. For instance, we obtain density estimates for Dirichlet improvability along typical points on a curve in Euclidean space. Other applications address approximations by algebraic numbers and best approximations (in the sense of Lagarias).
\end{abstract}

\maketitle
\setcounter{tocdepth}{1}
\tableofcontents

\section{Introduction}

\subsection{Dirichlet improvability and densities}\label{sec: dirichletimprovability}
Let $\xi = (\xi_1, \ldots, \xi_d)\in \R^d$ be a vector.
By Dirichlet's theorem on simultaneous approximations, the set of inequalities
\begin{align*}
  |p_1+q_1\xi_1+\ldots + q_d\xi_d| \leq N^{-d},
  \quad \max_i |q_i| \leq N
\end{align*}
has a non-zero solution $(p, q_1, \ldots, q_d) \in \Z^{d+1}$ for all $N \in \N$.
We say that $\xi$ is not $\mu$-Dirichlet improvable for $\mu\in (0,1)$ if there exist infinitely many $N \in \N$ for which the inequalities
\begin{align}\label{eq: Dirichletimprovable}
  |p_1+q_1\xi_1+\ldots + q_d\xi_d| \leq \mu N^{-d},
  \quad \max_i |q_i| \leq \mu N
\end{align}
have no non-zero solution $(p, q_1, \ldots, q_d) \in \Z^{d+1}$.
Davenport and Schmidt \cite{DavenportSchmidtII} proved that almost every $\xi \in \R^d$ is not $\mu$-Dirichlet improvable for any $\mu < 1$ (in the language of homogeneous dynamics, this also follows from the Dani correspondence \cite{Dani-correspondence} and recurrence -- see  \cite{KleinbockWeiss-Dirichletimprovable}).

Now let $\phi:[0,1]\to \R^d$ be a curve. If the image of $\phi$ is not contained in an affine hyperplane and $\phi$ is analytic, Shah \cite{Shah-translates1} proved that for almost every $s$ the vector $\phi(s)$ is not $\mu$-Dirichlet improvable for any $\mu<1$.
Among other results, we propose here a strengthening of Shah's result describing the logarithmic density of exceptions to Dirichlet improvability.
Define for any $\xi \in \R^d$ and $\mu \in (0,1)$
\begin{align*}
  \overline{\mathcal{D}}_{\xi}(\mu)  & = \limsup_{N' \to \infty}\frac{1}{\log(N')} \sum_{\substack{N \leq N', \\ \eqref{eq: Dirichletimprovable}\text{ holds for }N, \xi}} \frac{1}{N}, \\
  \underline{\mathcal{D}}_{\xi}(\mu) & = \liminf_{N' \to \infty}\frac{1}{\log(N')} \sum_{\substack{N \leq N', \\ \eqref{eq: Dirichletimprovable}\text{ holds for }N, \xi}} \frac{1}{N}.
\end{align*}

\begin{theorem}[Density in Dirichlet improvability]\label{thm: dirichletimprovable}
  There exists an explicit continuous strictly increasing function $f: [0,1] \to [0,1]$ with $f(0) = 0$ and $f(1) =1$ with the following properties.
  Suppose that $\phi:[0,1]\to \R^d$ is an analytic curve not contained in any affine hyperplane.
  Then for every $\mu\in (0,1)$ and for almost every $s \in [0,1]$ we have
  \begin{align*}
    \overline{\mathcal{D}}_{\phi(s)}(\mu) = \underline{\mathcal{D}}_{\phi(s)}(\mu) = f(\mu).
  \end{align*}
\end{theorem}

\begin{remark}
  It is a simple consequence of the Birkhoff ergodic theorem that there is a function $f$ as above so that for every $\mu\in (0,1)$ and for almost every $\xi \in \R^d$ we have $\overline{\mathcal{D}}_{\xi}(\mu) = \underline{\mathcal{D}}_{\xi}(\mu) = f(\mu)$.
  This function $f$ coincides with the one in Theorem~\ref{thm: dirichletimprovable} and is given by the measure of a parametrized family of neighborhoods of the cusp in $\SL_{d+1}(\R)/\SL_{d+1}(\Z)$.
  The speed of convergence $f(\mu) \to 1$ for $\mu \to 1$ has been studied in recent papers by Kleinbock and Yu \cite{KleinbockYu} for $d=1$ and Kleinbock, Str\"ombergsson, Yu \cite{KleinbockStrombergssonYu} for $d>1$. The decay of $f(\mu)$ for $\mu\to 0$ is easily established via Siegel's formula \cite{Siegel}.
\end{remark}

\begin{remark}
  The methods in Theorem~\ref{thm: dirichletimprovable} yield similar results for Diophantine approximation of matrices (see e.g.~\cite{ShahYang-translates, Yang-translates}).
\end{remark}

We turn now to algebraic approximations:
The \emph{polynomial height} $\|\alpha\|$ of an algebraic number $\alpha\in \R$ is defined to be the height $\height(\cdot)$ of its minimal polynomial, that is, the maximal absolute value of the coefficients for the minimal integer multiple of the minimal polynomial.
From Theorem~\ref{thm: dirichletimprovable}, one obtains the logarithmic density of solutions to
\begin{align*}
  P(s) \leq \mu N^{-d}, \
  P \in \Z[x] \text{ of degree }d \text{ with }\height(P)\leq \mu N
\end{align*}
for almost every $s \in [0,1]$.

Consider now the following problem. Given $s \in \R$, $\mu>0$, and $N \in \N$ does there exist an algebraic number $s' \in \R$ of degree $d$ with
\begin{align}\label{eq: algapprox}
  |s'-s|\leq \mu N^{-(d+1)} \text{ and }\norm{s'}\leq N \text{?}
\end{align}
Define
\begin{align*}
  \overline{\mathcal{A}}_s(\mu)  & = \limsup_{N' \to \infty}\frac{1}{\log(N')} \sum_{\substack{N \leq N', \\ \eqref{eq: algapprox}\text{ holds for }N, s}} \frac{1}{N}, \\
  \underline{\mathcal{A}}_s(\mu) & = \liminf_{N' \to \infty}\frac{1}{\log(N')} \sum_{\substack{N \leq N', \\ \eqref{eq: algapprox}\text{ holds for }N, s}} \frac{1}{N}.
\end{align*}

\begin{theorem}[Algebraic approximations]\label{thm: algapprox}
  There is an explicit continuous function $f: \R_{>0}\times \R \to (0,1)$ which is strictly increasing in the first coordinate with the following properties.
  Then for every $\mu\in (0,1)$ and for almost every $s \in \R$ we have
  \begin{align*}
    \overline{\mathcal{A}}_s(\mu) = \underline{\mathcal{A}}_s(\mu) = f(\mu, s).
  \end{align*}
\end{theorem}

\subsection{Best approximations of points on curves}\label{sec: best approx}
Fix a norm $\norm{\cdot}$ on $\R^d$ for $d \geq 2$ and let $\xi \in \R^d$.
Following Lagarias \cite{Lagarias1, Lagarias2} we say that a vector $(p, q) \in \Z^d \times \N$ is a best approximation of $\xi$ if
\begin{align*}
\norm{p-q\xi}< \min_{q'<q}\min_{p' \in \Z^d} \norm{p'-q'\xi} \text{ and } \norm{p-q\xi}=\min_{p' \in \Z^d} \norm{p'-q\xi}.
\end{align*}
If $\xi$ has an irrational coordinate, one obtains by recursion an infinite sequence $v_k = (p_k, q_k)$ of best approximations of $\xi$.
In a recent article \cite{ShapiraWeiss}, Shapira and Weiss obtain various new results about statistical properties of the best approximations of Lebesgue almost every $\xi$, some of which we now recall.

Define for any $v=(p, q) \in \Z^d \times \N$
\begin{align*}
  \disp(\xi, v) = q^{\frac{1}{d}}(p-q\xi).
\end{align*}
By Dirichlet's theorem, there exists a constant $C>0$ depending only on the norm $\norm{\cdot}$ so that $\norm{\disp(\xi, v)}\leq C$ whenever $v$ is a best approximation of $\xi$.
Moreover, we may associate a certain sequence of lattices with the best approximations. To that end, we identify $\R^d \simeq \{x \in \R^{d+1}: x_{d+1}=0\}$ and denote for any $v \in \R^{d+1}\setminus\R^d$ by $\pi^{v}: \R^{d+1}\to \R^d$ the projection parallel to $v$.
For $v \in \Z^{d}\times \N$ the projection $\Lambda_v = \pi^v(\Z^{d+1})$ defines a lattice in $\R^d$ and hence a point
\begin{align*}
  [\Lambda_v] \in X_d = \rquot{\SL_d(\R)}{\SL_d(\Z)}
\end{align*}
after dilation (not to be confused with the shape of the lattice $\Lambda_v$).

\begin{theorem}[{Shapira, Weiss \cite[Thm.~1.1]{ShapiraWeiss}}]\label{thm: ShapiraWeiss}
  For any norm $\norm{\cdot}$ on $\R^d$ there exists a Borel probability measure $\mu_{d, \norm{\cdot}}$ on $X_d \times\R^d$ such that for Lebesgue almost $\xi \in \R^d$ the following holds.
  Let $(v_k)$ be the sequence of best approximations of $\xi$. Then the sequence
  \begin{align*}
    \{([\Lambda_{v_k}], \disp(\xi, v_k))\}\subset X_d \times\R^d
  \end{align*}
  is equidistributed with respect to $\mu_{d, \norm{\cdot}}$.
\end{theorem}

The measure $\mu_{d, \norm{\cdot}}$ is further characterized in \cite{ShapiraWeiss} and is notably not a product measure. The projection of $\mu_{d, \norm{\cdot}}$ to $\R^d$ is boundedly supported and absolutely continuous with respect to Lebesgue with a non-trivial density. We refer to \cite{ShapiraWeiss} for a more precise description of the measure $\mu_{d, \norm{\cdot}}$ and for various extensions of the above theorem.
We remark that similar questions to Theorem~\ref{thm: ShapiraWeiss} had been studied previously -- cf.~\cite[\S3]{ShapiraWeiss} for a discussion.

As a consequence of the Birkhoff genericity results in the current article, we obtain results for the best approximations of points on curves.

\begin{theorem}[Best approximations of points on curves]\label{thm: bestapprox}
  Let $\phi:[0,1]\to \R^d$ be an analytic curve not contained in any affine hyperplane. Then for almost every $s \in [0,1]$ the following holds.
  Let $(v_k)$ be the sequence of best approximations of $\phi(s)$. Then the sequence
  \begin{align*}
    \{([\Lambda_{v_k}], \disp(\phi(s), v_k))\}\subset X_d \times\R^d
  \end{align*}
  is equidistributed with respect to $\mu_{d, \norm{\cdot}}$.
\end{theorem}

A similar result can be obtained for local properties of the best approximations (e.g.~the statistics of $v_k \mod M$ for any $M \geq 1$).

\subsection{Expanded translates of analytic curves in horospheres}\label{sec: intro-exptranslates}

Before turning to our main results, we briefly recall existing work on expanding translates of curves in horospheres.
For concreteness, let $G = \SL_d(\R)$ and let $\{a_t\} < G$ be a diagonalizable subgroup.
The expanding horospherical subgroup is then given by
\begin{align*}
  G^+ =\{g \in G: a_t g a_{-t} \to 1 \text{ as } t \to -\infty\}.
\end{align*}
A long line of research has established variants of equidistribution of expanded pieces of $G^+$-orbits in $\rquot{G}{\Gamma}$ for $\Gamma < G$ a lattice.
Here, equidistribution may for example be deduced by Margulis' thickening technique as well as mixing and can be made effective; this idea originated in Margulis' thesis \cite{Margulis-Thesis} (see also \cite[Prop.~2.4.8]{KleinbockMargulis-bounded}).
Motivated by this, one may ask under what conditions translates of a curve
\begin{align*}
  \varphi: [0,1] \to G^+
\end{align*}
(or more generally a submanifold) would equidistribute \cite{ShahICM}.
In his influential work \cite{Shah-translates1},
Shah resolved this question when $a_t = \mathrm{diag}(\euler^{(d-1)t}, \euler^{-t}, \ldots, \euler^{-t})$ and when the curve $\varphi$ is analytic and is not contained in any proper subgroup of $G^+$ (see also \cite{Shah-translatesSO, Shah-translatesSO-2}).
Subsequently, Shah and L.~Yang \cite{ShahYang-translates} and finally P.~Yang \cite{Yang-translates} generalized Shah's result to all diagonalizable subgroups.
Given the objectives of the current article, we shall state P.~Yang's result \cite{Yang-translates} here in the case
\begin{align}\label{eq: a_t}
  a_t = \mathrm{diag}(\euler^{t/m}, \ldots, \euler^{t/m}, \euler^{-t/n}, \ldots, \euler^{-t/n})\in G
\end{align}
where $m+n=d$.
In this case, obstructions to equidistribution may be described via constraining pencils introduced by Aka, Breuillard, de Saxc\'e, and Rosenzweig \cite{ABRS}.
Let $P< G$ be the parabolic group given by
\begin{align}\label{eq: parabolic}
  P = \{g \in G: \lim_{t \to \infty}a_tg a_{-t} \text{ exists}\}
\end{align}
and note that we may identify $P\backslash G \simeq \Gr(m, d)$.

\begin{definition}
  For an integer $r \leq m$ and a proper subspace $W \subset \R^{d}$ the pencil $\mathfrak{P}_{W, r}$ is
  \begin{align*}
    \mathfrak{P}_{W, r} = \{V \in \Gr(m, d): \dim(V\cap W) \geq r\}.
  \end{align*}
  The pencil is constraining if $\frac{\dim(W)}{r} < \frac{d}{m}$ and weakly constraining if  $\frac{\dim(W)}{r} \leq \frac{d}{m}$.
\end{definition}

For example, when $m=1$ the pencil $\mathfrak{P}_{W,1}$ consists of all lines contained in $W$ and is always constraining.
If $m=n=2$ the pencil $\mathfrak{P}_{W, r}$ is weakly constraining if $\dim(W)=2$ and $r=1$, constraining if $\dim(W) \in \{1,2,3\}$ and $r=2$, and not constraining if $\dim(W) = 3$ and $r=1$.

\begin{definition}
  A partial flag subvariety of $P\backslash G$ (for $\{a_t\}$) is a subvariety of the form $PHg$ where $H < G$ is reductive group containing $\{a_t\}$ and where $g \in G$.
\end{definition}

\begin{theorem}[{P.~Yang \cite{Yang-translates}}]\label{thm: Yang}
  Let $\Gamma < G$ be a lattice.
  Let $\{a_t: t \in \R\}$ be as in \eqref{eq: a_t} and $\varphi: [0,1] \to G$ be an analytic curve whose image in $P\backslash G$ is not contained in a weakly constraining pencil or in a partial flag subvariety. Then for any $x \in G/\Gamma$ and any  $f\in C_c(G/\Gamma)$
  \begin{align*}
    \int_0^1 f(a_t \varphi(s)x) \de s\to \int_{G/\Gamma} f\de m_{G/\Gamma} \quad (t\to \infty).
  \end{align*}
\end{theorem}

The same conclusion was obtained in earlier work by Shah and L.~Yang \cite{ShahYang-translates} under a stronger assumption called supergenericity.
The conditions on $\varphi$ are not to be considered sharp:
Suppose that $m=1$ in which case $G^+ \simeq \R^{d-1}$ and the above conditions for $\varphi: [0,1] \to G^+$ translate to the image of $\varphi$ not being contained in any affine hyperplane; Theorem~\ref{thm: Yang} in this case is due to Shah \cite{Shah-translates1}.
For $d=3$ work by Shi and Weiss \cite{ShiWeiss} (with an additional average over $t$) and Kleinbock, de Saxc\'e, Shah, and P.~Yang \cite{KSSY} addresses conditions under which translates of planar lines equidistribute.
Similar results exist for $d>3$ \cite{ShahYang-hyperplanes}. In the case $m=1$, the analyticity assumption has been weakened recently by Shah and P.~Yang \cite{ShahYang-smooth}.

\subsection{Birkhoff genericity and the main result}
In this article, we establish that almost every point on a curve in an expanding horosphere is Birkhoff generic for the diagonalizable flow under suitable conditions.

In view of Birkhoff's ergodic theorem, one may ask whether for an ergodic measure preserving system $\mathsf{X} = (X, \mathcal{B}, \mu, T)$ the ergodic averages $\frac{1}{N}\sum_{n=1}^N f(T^nx)$ converge for almost every point $x$ when the point $x$ is sampled with respect to a different measure.
Assuming that $X$ is locally compact $\sigma$-compact, we say that a point $x$ is Birkhoff generic (with respect to $\mu$ and $T$) if
\begin{align*}
  \frac{1}{N}\sum_{n=1}^N f(T^nx) \to \int f \de \mu
\end{align*}
for any $f \in C_c(X)$. By the Birkhoff ergodic theorem, $\mu$-almost every point is Birkhoff generic with respect to $\mu$.
The above question for $X = \mathbb{T}= \R/\Z$ has been studied by various authors including Host \cite{Host-normal} and Hochman and Shmerkin \cite{HochmanShmerkin}. Host established that $\nu$-almost every point is Birkhoff generic for $\times 2$ when $\nu$ is a $\times 3$ invariant and ergodic measure $\nu$ on $\mathbb{T}$ of positive entropy.
For example, almost any point on the middle-third Cantor set is Birkhoff generic for $\times 2$ i.e.~normal in base $2$ \cite{Schmidt-normal, Cassels-normal}.

Now let $\{a_t: t\in \R\}$ be as in \eqref{eq: a_t} and let $G^+$ be the expanding horospherical subgroup.
Birkhoff genericity for fractal measures on $G^+$ is known in many cases -- see \cite{EinsiedlerFishmanShapira, SimmonsWeiss, ProhaskaSert, ProhaskaShiSert}.
For certain subgroups of $G^+$ (and for general diagonalizable subgroups), Shi \cite{shi2020pointwise} established Birkhoff genericity.
Moreover, Fraczek, Shi, and Ulcigrai \cite{FraczekShiUlcigrai} and Zhang \cite{Zhang} proved Birkhoff genericity for certain submanifolds of the full horosphere in $\SL_d(\R) \ltimes (\R^d)^k$ which project surjectively onto the full horosphere in $\SL_d(\R)$.
All of the works \cite{shi2020pointwise, FraczekShiUlcigrai, Zhang}
are based on previous work of Chaika and Eskin \cite{ChaikaEskinBirkhoffGeneric}.

In the current article, we establish the following Birkhoff genericity result.

\begin{theorem}[Birkhoff genericity along curves]\label{thm: main}
  Let $G = \SL_d(\R)$, $\Gamma = \SL_d(\Z)$, $X = G/\Gamma$, $\{a_t: t \in \R\}$ as in \eqref{eq: a_t}, and $P$ the associated parabolic subgroup as in \eqref{eq: parabolic}.
  Let $\varphi:[0,1] \to G$ be an analytic curve whose image in $P\backslash G$ is not contained in any weakly constraining pencil or in any partial flag subvariety. Then for any $x \in X$ and for almost every $ s\in [0,1]$ the point $\varphi(s)x$ is Birkhoff generic with respect to the Haar probability measure on $X$ and the flow $\{a_t: t >0\}$ i.e.~for every $f\in C_c(X)$
  \begin{align*}
    \frac{1}{T} \int_0^T f(a_t \varphi(s)x) \de t \to \int_X f \de m_X.
  \end{align*}
\end{theorem}

\begin{remark}
  The non-degeneracy assumptions on the curve $\varphi$ in the case $m=1$ in either of the above results translates to assuming that the image of $\varphi$ in $G/P\simeq \mathbb{P}^d$ is not contained in any proper subspace.
\end{remark}

\begin{remark}
  We do not expect the conditions in either of the above theorems to be sharp and refer to \S\ref{sec: intro-further} for a discussion.
  In the upcoming work of Shah and P.~Yang Theorem~\ref{thm: main} for degenerate curves is established in some important cases.
\end{remark}

\begin{remark}
  The lattice $\SL_d(\Z)$ may be replaced by a congruence lattice $\Gamma< \SL_d(\Z)$ in the above without difficulty.
\end{remark}

\textbf{Acknowledgments:}
The authors would like to express their gratitude to Barak Weiss for suggesting this fascinating topic. 
We are also deeply thankful to Elon Lindenstrauss for insightful conversations and for generously sharing his ideas. We greatly appreciate the helpful comments from the anonymous referees. Finally, we would like to thank Roland Prohaska and Pengyu Yang for valuable discussions related to this work.

\subsection{Ingredients of the proof}\label{sec: ingredients}
The proof consists of roughly four main ingredients, the first three not being novel:
\begin{itemize}
  \item \textbf{Establishing unipotent invariance:} This is a general tool, first introduced in \cite[Prop.~3.1]{ChaikaEskinBirkhoffGeneric} and explored more in \cite[Prop.~2.2]{shi2020pointwise}. It shows that limits of Birkhoff averages starting at a suitable expanding set are invariant under unipotent directions, which are determined from local information of the starting point. Both \cite{ChaikaEskinBirkhoffGeneric, shi2020pointwise} use effective estimates of correlations to prove this unipotent invariance. 
  We rephrase the proof with martingales, as this is a recurring theme in the paper.
  \item \textbf{Linearization:}
        Provided that a measure $\mu$ is invariant under a unipotent subgroup, Ratner's measure classification \cite{ratner91-measure}
        implies that the ergodic components are of algebraic nature. Linearization is a general technique, introduced by Dani and Margulis \cite{DaniMargulis93}, transforming the problem of showing that $\mu$ is the Haar measure to countably many algebraic-like problems. These problems can be summarized as showing that ``some vectors do not spend a lot of time in a certain set, with high probability.''
        For the purposes of this overview, we omit discussing here the appearance of the (related) quantitative non-divergence techniques of Kleinbock and Margulis \cite{KleinbockMargulis98}.
  \item \textbf{Linear Dynamics:} The linearization technique requires that one rules out (among other things) the existence of a nonzero vector $v \in \R^d$ such that $\sup_{s \in [0,1]}\norm{a_t \varphi(s) v} \to 0$.
        Statements of this flavor have been established in the works of Shah \cite{Shah-translates1}, Shah, L.~Yang \cite{ShahYang-translates} and Yang \cite{Yang-translates} (usually called 'basic lemma') in various instances.
\end{itemize}
\textbf{Interlude, Combinatorial structure of the problems:}
For the sake of this introduction, we provide a combinatorial structure of the problems the linearization technique yields.
Let $\rho: \SL_d \to \SL_N$ be a representation defined over $\Q$.
For each time $t$ we have a collection $A_t$ of disjoint intervals $I\subset [0, 1]$, each of which is a connected component of a set of the form $\{s\in [0,1]: \norm{\rho(a_t\varphi(s))v_I}\leq 1\}$ for $v_I \in \Z^N$ nonzero.
For $\varepsilon>0$ let
\begin{align*}
  \mathrm{Bad}_t(\varepsilon) = \bigcup_{I \in A_t} \{s \in I: \norm{\rho(a_t\varphi(s))v_I}\leq \varepsilon\}
\end{align*}
The goal is to bound the amount of $t\in [0, T]$ for which $s\in \mathrm{Bad}_t(\varepsilon)$ in terms of $\varepsilon$ almost surely for $s\in [0,1]$.
\begin{itemize}
  \item \textbf{Interval combinatorics:}
        We first employ an argument by Kleinbock and Margulis \cite{KleinbockMargulis98} in their work on quantitative non-divergence, and obtain a bound on the total volume of exceptionally long intervals, that is,
        $$\left|\bigcup_{\substack{I\in A_t\\|I|>\euler^{-rt + t'}}}I\right| \ll \euler^{-\delta t'},$$
        for some $\delta>0$ and $t'< \delta t$.
        By the Borel–Cantelli lemma one may hence ignore the intervals $\{I\in A_t:|I|>\euler^{-rt + \sqrt t}\}$ (from now on these are called \emph{long intervals}). The Short Intervals Lemma \ref{lem: martingale interval lemma} is used to bound the rest (these are called \emph{short intervals}).
        The Short Intervals Lemma takes the place of Margulis functions, which are used for instance in \cite{ChaikaEskinBirkhoffGeneric, shi2020pointwise}.
\end{itemize}
\subsection{Further research}\label{sec: intro-further}
One can try to extend the setting of the problem into a more general one:
\begin{problem}
Extend Theorems \ref{thm: main} to non-abelian horospheres (cf.~Yang \cite{Yang-translates}).
\end{problem}
One obstruction in using our technique occurs when ruling out long intervals (Proposition~\ref{prop: excep int}).
For instance, when $G^+$ is abelian, a curve $s \mapsto a_t \varphi(s)v$ through a vector $v \in\R^d$ can be described in a small neighborhood of a point $s_0 \in [0,1]$ through a unipotent curve $\exp(x Y)w$ with $Y$ nilpotent and $w \in \R^d$. The same is not true when the horosphere is non-abelian.

Similarly, one may ask for the following extension:

\begin{problem}
Extend Theorems~\ref{thm: main} to smooth curves (cf.~Shah and Yang \cite{ShahYang-smooth}).
\end{problem}

In the spirit of Shi's work \cite{shi2020pointwise} (see also \cite[Thm.~1.10]{Shah-translates1}) we pose the following problem:

\begin{problem}
Let $\rho: G=\SL_d(\R) \to L$ be a representation of Lie groups and let $\Gamma< L$ be a lattice.
Let $\varphi:[0,1] \to G$ be an analytic curve whose image in $P\backslash G$ is not contained in any weakly constraining pencil or in any partial flag subvariety.
Let $x \in L/\Gamma$.
Show that for almost every $s \in [0,1]$ and any $f \in C_c(L/\Gamma)$
\begin{align*}
  \frac{1}{T}\int_0^T f(\rho(a_t\varphi(s))x) \de t \to \int f \de m_{\overline{\rho(G)x}}
\end{align*}
where $m_{\overline{\rho(G)x}}$ is the invariant probability measure on the orbit closure of $x$.
\end{problem}

In the spirit of \cite{EinsiedlerFishmanShapira, SimmonsWeiss,
  ProhaskaSert, ProhaskaShiSert} one may ask the following:

\begin{problem}
Can one prove Theorem \ref{thm: main} for friendly measures (see \cite{KleinbockLindenstraussWeiss}), in place of the uniform measure on a curve?
\end{problem}

One could also consider non-conventional ergodic averages (see e.g.~the work of Bourgain \cite{Bourgain1, Bourgain2} for polynomial averages):

\begin{problem}
Let $\{t_n\}$ be a sequence of integers.
Determine sufficient conditions on the curve $\varphi$ and the sequence $\{t_n\}$ so that for all $x \in X$ and for almost every $s \in [0,1]$
\begin{align*}
  \frac{1}{N}\sum_{n=1}^N f(a_{t_n}\varphi(s)x) \to \int_X f \de m_X
\end{align*}
for all $f \in C_c(X)$.
\end{problem}

A different type of questions lies in changing the methods of our proof:
\begin{problem}
Can one extend the argument given in \S\ref{sec: unipotent inv} to prove invariance in more directions, not only the tangent one?
\end{problem}
This would nullify the need to use the linearization technique, though one would still have to prove non-escape of mass.

\subsection{Structure of the article}
In \S\ref{sec: notations and preliminaries} we define the notations that we use and introduce preliminary tools on $(C, \alpha)$-good functions.
In \S\ref{sec: long inter} we use elementary algebraic manipulations combined with quantitative non-divergence \cite{KleinbockMargulis98} and results in linear dynamics \cite{Shah-translates1, Yang-translates} to control the `long intervals' in the sense of \S\ref{sec: ingredients}. In \S\ref{sec: short intervals lemma} we prove a completely probabilistic lemma which will help us control the `short intervals'.
In \S\ref{sec: nondivergence} we combine the previous tools and give a demonstration of the method, proving a (partially quantitative) nondivergence result.
In \S\ref{sec: unipotent inv} establish unipotent invariance of limits the Birkhoff averages. It is independent of the previous sections, and it uses probabilistic methods similar to those in \S\ref{sec: short intervals lemma}.
In \S\ref{sec: avoid sets and proof} we combine the previous results with Ratner's measure classification \cite{ratner91-measure} and the linearization technique \cite{DaniMargulis93} to prove Theorem \ref{thm: main}.
Finally, in \S\ref{sec: applications} we establish the applications in \S\ref{sec: dirichletimprovability}, \ref{sec: best approx}.

\section{Notation and preliminaries}
\label{sec: notations and preliminaries}
As in Theorem~\ref{thm: main}, $\{a_t: t \in \R\}$ shall be the (real) diagonalizable subgroup of $G = \SL_d(\R)$, $d \geq 3$ defined in \eqref{eq: a_t}.
Let $P$ be the parabolic subgroup in \eqref{eq: parabolic} and let $\mathfrak{p}$ be its Lie algebra.

Let $Z = Z(\{a_t\})<P$ be the centralizer subgroup of $\{a_t\}$.
Note that $Z$ acts by conjugation on the expanded horospherical subgroup
\begin{align*}
  G^+ = \{g \in G: a_t g a_{-t} \to \id \text{ as } t \to -\infty\}.
\end{align*}
For $\{a_t\}$ as above, one may identify
$G^+ \simeq \Lie(G^+) = \Fg^+ \simeq \Mat_{mn}(\R)$ and
\begin{align*}
  Z \simeq & (\GL_{m}(\R) \times \GL_n(\R))^1                                               \\
           & \quad:= \{(g_1, g_2)\in \GL_{m}(\R) \times \GL_n(\R): \det(g_1)\det(g_2) = 1\}
\end{align*}
where the action of $Z$ on $G^+$ corresponds to the action of $(\GL_{m}(\R) \times \GL_n(\R))^1$ on $\Mat_{mn}(\R)$ via $(g_1, g_2).X = g_1 Xg_2^{-1}$.

\begin{example}
  If $a_t = (\euler^{(d-1)t}, \euler^{-t}, \ldots, \euler^{-t})$ then $G^+ \simeq \R^{d-1}$, $Z \simeq \GL_{d-1}(\R)$, and the action of $Z$ on $G^+$ corresponds to the (non-standard) action of $\GL_{d-1}(\R)$ on $\R^{d-1}$ via $v \mapsto g.v = \det(g)^{-1}(g^{-1})^t v$.
\end{example}

\begin{lemma}\label{lem: centralizer Ad-orbits}
  There are real subvarieties $V_0 \subset V_1 \subset \ldots \subset V_{\min\{m, n\}}= \mathfrak{g}^+$ so that each $V_k$ is invariant under $Z$ and the action of $Z$ on $V_k \setminus V_{k-1}$ has at most two orbits both of which are open in $V_k$ (for the Hausdorff topology).
\end{lemma}

\begin{proof}
  Under the identification $\mathfrak{g}^+ \simeq \Mat_{mn}(\R)$ the action of $Z$ preserves the rank.
  For each $k \leq \min\{m, n\}$ let $V_k \subset \mathfrak{g}^+$ be the subvariety of matrices of rank at most $k$.
  Fix a rank $k \leq \min\{m, n\}$ and consider the action of $Z$ on $V_k \setminus V_{k-1}$; we prove that it has at most two orbits\footnote{A more careful argument show that the action of $Z$ on $V_k\setminus V_{k-1}$ is transitive if $k < \max\{m, n\}$ and has two orbits if $k=m=n$.}.
  A section for this action is given by
  \begin{align*}
    S_k = \left\{ \begin{pmatrix}
                    A & 0_{k, n-k} \\ 0_{m-k, k} & 0_{m-k, n-k}
                  \end{pmatrix}: A \in \GL_{k}(\R)\right\} \subset V_k
  \end{align*}
  where a subgroup of $Z$ identifiable with $(\GL_k(\R)\times \GL_k(\R))^1$ still acts on $S_k$ (in the natural manner). This new action is transitive on matrices $A$ whose determinants share the same sign.
\end{proof}

We let $\varphi:[0,1] \to G$ be an analytic curve.
We shall often assume that the image of $\varphi$ in $P\backslash G$ is not contained in any weakly constraining pencil.
Equivalently, the image of $\varphi$ in $P\backslash G$ is not contained in any weakly unstable Schubert variety \cite[Props.~6.4, 6.5]{Yang-translates}.

In the following we shall always write $r = \frac{1}{m}+ \frac{1}{n} >0$ for the expansion rate of $G^+$ (that is, the eigenvalue of $\Ad(a_t)$ acting on $\Fg^+$ is $\euler^{rt}$).
We define $\Phi^+(s, \xi)\in \Fg^+$ and $\Phi^{0-}(s, \xi)\in \mathfrak{p}$ for $s \in [0,1]$ and $\xi$ sufficiently small through
\begin{align}\label{eq:defPhi0-,+}
  \varphi(s+\xi)\varphi(s)^{-1} = \exp(\Phi^{0-}(s, \xi))\exp(\Phi^{+}(s, \xi)).
\end{align}
We remark that $\Phi^{0-}(\cdot,\cdot)$ and $\Phi^{+}(\cdot,\cdot)$ are analytic in both variables.
Moreover, we have $\Phi^{0-}(s, \xi) = O(\xi)$ and
\begin{align}\label{eq:defYs}
  \Phi^{+}(s,\xi) = \xi Y_s + O(\xi^2)
\end{align}
for some analytic curve $s \in [0,1] \to Y_s$.

\begin{lemma}\label{lem:Yneq0}
  If the image of $\varphi$ in $P\backslash G$ is non-constant, the curve $s \mapsto Y_s$ is non-zero at all but finitely many points.
\end{lemma}

\begin{proof}
  As $s \mapsto Y_s$ is analytic, we may assume by contradiction that $Y_s=0$ for all $s \in [0,1]$. Then for $\xi$ small the distance between $P\varphi(s)$ and $P\varphi(s+\xi)$ is of order $O(\xi^2)$ as
  \begin{align*}
    P\varphi(s+\xi) = P(I+O(\xi^2))\varphi(s) =P\varphi(s)(I+O(\xi^2))
  \end{align*}
  by \eqref{eq:defYs}. In particular, the derivative of $s \mapsto P\varphi(s)$ is zero everywhere and hence $s \mapsto P\varphi(s)$ is constant.
\end{proof}

The following lemma is a direct consequence of \eqref{eq:defPhi0-,+} and \eqref{eq:defYs}.

\begin{lemma}\label{lem: directionalasymp}
  For any $\xi \in \R$ sufficiently small and all $t >0$ we have
  \begin{align*}
    a_t \varphi(s+\xi) = (I+O(\xi+\xi^2\euler^{rt}))\exp(\xi \euler^{rt}Y_s)a_t \varphi(s).
  \end{align*}
  In particular, the distance between $a_t \varphi(s+\xi)$ and $a_t \varphi(s)$ is of order $O(\xi\euler^{rt})$.
\end{lemma}

Let $k \leq \min\{m, n\}$ be minimal such that $V_k \cap \{Y_s: s \in [0,1]\}$ is infinite.
By Lemma~\ref{lem:Yneq0}, we have $k>0$.
By analyticity, we have $Y_s \in V_k$ for all $s \in [0,1]$.
Since $V_{k-1} \cap \{Y_s: s \in [0,1]\}$ is finite by construction, we may assume without loss of generality that $Y_s \not\in V_{k-1}$ for all $s \in [0,1]$ by a covering argument.
In particular, by Lemma \ref{lem: centralizer Ad-orbits} the points $\{Y_s: s\in [0,1]\}$ lies on the same $Z$ orbit. Hence there exists an analytic curve $s \mapsto \rot(s) \in Z$ such that
\begin{align}\label{eq: choice of z()}
  \rot(s).Y_s = Y_0
\end{align}
for all $s \in [0,1]$.
We note that after the above reduction $Y_s \neq 0$ for all $s \in [0,1]$.

\begin{remark}
  It is not generally necessary (cf.~\cite{Yang-translates}) to map $Y_s$ to, say, $Y_{s'}$ within the centralizer of the flow and indeed it is only possible when the horosphere is abelian (or equivalently minimal). In the context of this article, it is however convenient.
\end{remark}

\subsection{\texorpdfstring{$(C, \alpha)$}{(C, alpha)}-good functions}\label{sec: Calpha}
We recall the following (by now standard) definition from \cite{KleinbockMargulis98}:
\begin{definition}[$(C, \alpha)$-good functions]
  Let $I \subset \R$ be an interval and let $f: I \to \R$.
  Given $C>0$ and $\alpha>0$ we say that $f$ is $(C, \alpha)$-good if for any subinterval $J\subset I$ and any $\varepsilon>0$ we have
  \begin{align*}
    |\{x \in J: |f(x)|< \varepsilon\}|\leq C \Big(\frac{\varepsilon}{\sup_{x \in J}|f(x)|}\Big)^\alpha |J|.
  \end{align*}
  Here (and henceforth), $|A|$ for a subset $A \subset \R$ indicates the Lebesgue measure of~$A$.
\end{definition}

For example, polynomials of degree at most $n$ are $(C_n, \frac{1}{n})$-good for some $C_n>0$ (cf.~\cite[Lemma 4.1]{DaniMargulis93}, \cite[Prop.~3.2]{KleinbockMargulis98}).

For any $\gamma>0$ and $J=[a-b, a+b] \subset \R$ an interval define $\gamma.J = [a-\gamma b, a+\gamma b]$ the centered $\gamma$-dilation of $J$.

\begin{lemma}\label{lem:dilation}
  Let $f:[0,1] \to \R$ be $(C, \alpha)$-good for some $(C, \alpha)$ and let $I \subset [0,1]$ be a maximal subinterval with $\sup_{s \in I}|f(s)|\leq 1$.
  If $J \subset I$ is an interval with $\sup_{s \in J}|f(s)|\leq \varepsilon$ then $\frac{1}{C\varepsilon^\alpha}.J \cap [0,1] \subset I$.
\end{lemma}
\begin{proof}
  Write $I = [a, b]$ and $J = [s_0-c, s_0+c]$. By the maximality condition of $I$, we get that either $|f(a)| = 1$ or $a = 0$. Similarly, either $|f(b)| = 1$ or $b = 1$.
  To show the desired containment, it is sufficient to show that
  \[\left[ s_0 - \frac{1}{C\varepsilon^\alpha}c, s_0\right] \cap [0,1] \subseteq I\qquad \text{and}\qquad\left[ s_0, s_0 + \frac{1}{C\varepsilon^\alpha}c\right] \cap [0,1] \subseteq I.\]
  We will show the first containment, and the second will follow similarly.
  If $a=0$ the containment follows since $s_0\in I$. Otherwise, $|f(a)|=1$.
  Applying the $(C, \alpha)$-property on $I^- = [a, s_0]$ yields that the measure of $J\cap I^{-}$ relative to $I^{-}$ is at most $ C\varepsilon^\alpha$ which implies the desired dilation statement.
\end{proof}

The following local to global property for $(C,\alpha)$-good functions is already mentioned in \cite[\S2.1]{Shah-translatesSO}, we include a full proof for lack of reference.

\begin{lemma}[Locality of the $(C, \alpha)$-goodness property.]\label{lem: C alpha local}
  Let $I$ be an interval and $f: I \to \RR$ be a function. Suppose that $I = \bigcup_{i=1}^k I_i$ is a union of intervals such that $f$ is $(C, \alpha)$ good on each interval $I_i$ for $C>1$ and $\alpha>0$, and for every $i,j \in \{1,\ldots,k\}$ the intersection $I_{i}\cap I_{j}$ is either empty or has length $|I_{i}\cap I_{j}| \ge \eta$.
  Then $f$ is $\Big(C^kk\big(\frac{|I|}{\eta}\big)^{k-1}, \alpha \Big)$-good.
\end{lemma}

\begin{proof}
  Let $J \subset I$ be a subinterval. We will prove the goodness property with respect to this interval.
  For any $i \in \{1,\ldots,k\}$ set $J_i = I_i \cap J$.
  In the following we shall call two indices $i,j \in \{1,\ldots,k\}$ paired if $I_i \cap I_j \neq \emptyset$.
  Whenever $i,j$ are paired, we have either $J_i \cap J_j = I_i \cap I_j$ or one of the interval $J_i,J_j$ may be discarded in the induced cover of $J$.
  Indeed, if the intersection $I_i\cap I_j$ does not intersect $\overline{J}$ one of the intervals $J_i,J_j$ is empty. 
  Also, if it contains a boundary point of $J$ then either $J_i \subset J_j$ or $J_j \subset J_i$.
  After discarding intervals, we may assume without loss of generality that $I_i\cap I_j = J_i \cap J_j$ for all $i,j$.

  Now suppose that $i,j$ are paired. We claim that
  \begin{align}\label{eq:pairedint}
    \sup_{x \in J_i}|f(x)| \geq \tilde{\eta} \sup_{x \in J_j}|f(x)|
  \end{align}
  where $\tilde{\eta} = (\eta/(C|I|))^{\frac{1}{\alpha}}$. Indeed, as $f$ is $(C,\alpha)$-good on $I_j$ we have
  \begin{align*}
    \eta \leq |J_i \cap J_j| \leq \big|\big\{x \in J_j:|f(x)|\leq \sup_{x\in J_i}|f(x)|\big\}\big|
    \leq C \bigg(\frac{\sup_{x\in J_i}|f(x)|}{\sup_{x\in J_j}|f(x)|}\bigg)^\alpha|I|
  \end{align*}
  which implies \eqref{eq:pairedint}.

  Note that $\sup_{x\in J}|f(x)| = \sup_{x\in J_{i_0}}|f(x)|$ for some $i_0$.
  Consider the graph whose vertices are indices $i=1,\ldots, k$ and whose edges are $(i,j)$ where $i,j$ are paired.
  It is connected since the intervals $J_i$ cover the connected set $J$.
  Iterating \eqref{eq:pairedint} over a path between $i_0$ to any other vertex $i$ yields that for every $i=1,\ldots, k$ we have $\sup_{x\in J_i}|f(x)| \ge \tilde{\eta}^{k-1} \sup_{x\in J}|f(x)|$.
  Finally, for every $\varepsilon>0$ we have
  \begin{align*}
    |\{x \in J: |f(x)| \leq \varepsilon\}|
     & \leq \sum_{i=1}^k |\{x \in J_i: |f(x)| \leq \varepsilon\}|                                                    \\
     & \leq C \sum_{i=1}^k \Big(\frac{\varepsilon}{\sup_{x\in J_i}|f(x)|}\Big)^\alpha |J_i|                          \\
     & \leq C\tilde{\eta}^{-\alpha(k-1)} \sum_{i=1}^k \Big(\frac{\varepsilon}{\sup_{x\in J}|f(x)|}\Big)^\alpha |J_i| \\
     & \leq Ck\tilde{\eta}^{-\alpha(k-1)} \Big(\frac{\varepsilon}{\sup_{x\in J}|f(x)|}\Big)^\alpha |J|
  \end{align*}
  as desired.
\end{proof}

The following lemma in particular establishes uniform $(C, \alpha)$-goodness for certain families of analytic functions.

\begin{lemma}\label{lem: C, alpha good}
  Suppose that $\varphi:[0,1]\to M_{N\times N}(\R)$ is analytic.
  Then there exist constants $C > 0$, $\alpha>0$, and $M>0$ such that for any $v \in \R^N$ and $g \in M_{N\times N}(\R)$ the following hold:
  \begin{itemize}
    \item The function $s \mapsto \norm{g\varphi(s)v}$ is $(C, \alpha)$-good.
    \item For any $x>0$ the set
          \begin{align*}
            \{s \in [0,1]: \norm{g\varphi(s)v}< x\}
          \end{align*}
          has at most $M$ connected components.
  \end{itemize}
\end{lemma}

\begin{proof}
  Since $f$ is $(C, \alpha)$-good if and only if $f^2$ is $(C, \alpha/2)$-good, we may prove the lemma for $s \mapsto \norm{g\varphi(s)v}^2$ instead.
  Let $\mathcal{F}$ be the linear span of the products of all pairs of entry functions of $\varphi$. The space $\mathcal{F}$ is a finite-dimensional vector space as it is generated by $N^2(N^2+1)/2$ elements.
  The map $s\mapsto \norm{g\varphi(s)v}^2$ is in $\mathcal F$ for every $t\in \RR$, $v\in \R^N$ and $g \in M_{N\times N}(\R)$, hence we will prove the conclusions of the lemma for all $f \in \mathcal{F}$.

  Let $S_{\mathcal{F}}$ be a unit sphere in $\mathcal{F}$ (with respect to any norm).
It is sufficient to prove the conclusions of the lemma for functions in $S_\mathcal{F}$ since any non-zero function in $\mathcal{F}$ is a multiple of a function in $S_\mathcal{F}$ (by normalizing).

  For every $p=(f, s)\in S_{\mathcal{F}} \times [0,1]$ there is $k=k_p\in \NN$ such that $|f^{(k)}(s)| > 0$ by analyticity of $f$.
  We can find a neighborhood $U_p\times I_p$ of $p=(f, s)\in S_{\mathcal{F}}\times [0,1]$ and $\varepsilon_p>0$ such that for every $p_0 = (f_0, s_0)\in U_p\times I_p$ we have $|f^{(k)}_0(s_0)|>\varepsilon_p$.
  Moreover, we may assume that $I_p$ is always an interval.
  Since $S_{\mathcal{F}}\times [0,1]$ is compact we deduce that there is a finite collection $P\subset S_{\mathcal{F}}\times [0,1]$ such that $\{U_p\times I_p: p\in P\}$ covers $S_{\mathcal{F}}\times [0,1]$.

  By the mean value theorem if $f\in U_p$ then $\{s\in I_p:|f(s)|<x\}$ has at most $k_p$ connected components for any $x >0$. The second part of the lemma follows from the finiteness of $P$.
  For the first part of the lemma we may apply \cite[Lemma 3.3]{KleinbockMargulis98} and obtain that for any $f \in U_p$ the restriction $f|_{I_p}$ is $(C_p, \alpha_p)$-good for some $C_p, \alpha_p>0$ (depending on $k_p$).
  Thus, the first part of the lemma follows also from the finiteness of $P$ and Lemma \ref{lem: C alpha local}.
\end{proof}

\section{Long intervals}
\label{sec: long inter}

\subsection{Contraction for long intervals with short vectors}

\begin{claim}
  \label{claim: polynomial size}
  Let $P(x) = a_0x^n+a_1x^{n-1}+a_2x^{n-2} +\cdots + a_n$ be a polynomial. Then for every $s>0$ we have
  \begin{align*}
    \max_{x\in [0, s]}|P(x)| \asymp_n \max_{k=0, \ldots, n}|a_ks^{n-k}|.
  \end{align*}
\end{claim}
\begin{proof}
  Using rescaling it is sufficient to prove the lemma for $s=1$ then the claim follows from the fact that both sides are norms in the coefficient vector.
\end{proof}
\begin{claim}\label{claim: unipotent homomorphisms bound}
  Let $h\in \mathfrak{sl}_N(\RR)\setminus \{0\}$ be a nilpotent matrix.
  Then there is a diagonalizable flow $t \mapsto b_t = \exp(b't)\in \SL_N(\RR)$ (where $b' \in \mathfrak{sl}_N(\RR)$) and a constant $\delta \geq \frac{1}{N}$ such that for every $t_0\ge 0$, vector $v\in \RR^N$ and
  interval $I\subseteq \RR$ with $|I|\ge \euler^{t_0}$ we have
  \[
    \max_{s\in I} \|b_{t_0}\exp(sh) v\| \asymp_{h} \euler^{-\delta t_0}\max_{s\in I}\|\exp(sh) v\|.
  \]
\end{claim}
\begin{proof}
  First, we will prove the claim whenever
  $h$ is a nontrivial Jordan block,
  $$h = \begin{pmatrix}
      0 & 1 &        &   \\
        & 0 & \ddots &   \\
        &   & \ddots & 1 \\
        &   &        & 0
    \end{pmatrix}.$$
  Let $b(t) = \exp\left(t \diag \left(\frac{1-N}{2}, \frac{3-N}{2}, \dots, \frac{N-1}{2}\right)\right)$.
  We may shrink $I$ and assume that $|I| = \euler^{t_0}$. By changing $v$ we may assume that $I = [0, \euler^{t_0}]$.
  Now if $v = (v_i)_{i=0}^{N-1}$ then by Claim \ref{claim: polynomial size} the maximal $k$-th coordinate satisfies
  \begin{align*}
    \max_{s\in I}|(\exp(sh)v)_k| \asymp_N \max_{i=0,\ldots,N-1-k} \euler^{it_0}|v_{i+k}|.
  \end{align*}
  It follows that
  \[\max_{s\in I}\|\exp(sh) v\| \asymp_N \max_{i=0,\ldots,N-1} \euler^{it_0}|v_{i}|.\]
  Similarly
  \begin{align*}
    \max_{s\in I} \|b(t_0)\exp(sh) v\|
     & \asymp_N \max_{k=0,\ldots,N-1}\big(\euler^{(1+2k-N)t_0/2}\max_{i=0,\ldots,N-1-k} \euler^{it_0}|v_{i+k}|\big) \\
     & \asymp_N \euler^{(1-N)t_0/2}\max_{i=0,\ldots,N-1} \euler^{it_0}|v_{i}|.
  \end{align*}
  The result follows for $\delta = \frac{N-1}{2}$.
  Now, whenever $h$ is of Jordan normal form, say, $h = \diag(h_0, \dots, h_{k})$ where $h_i\in \lSL_{N_i}(\RR)$ is a Jordan block.
  Let $b_0(t), \dots, b_{k}(t)$ be the corresponding homomorphisms and $\delta_i = \frac{N_i-1}{2}$ for $i=0, \dots, k$.
  Let $\delta = \frac{1}{N}\sum_{i=0}^{k} N_i \delta_i$ be the weighted average of the $\delta_i$. It satisfies that
  $\sum_{i=0}^k N_i (\delta-\delta_i) = 0$.
  Then $b(t) = \diag(\euler^{-(\delta-\delta_i)t}b_i(t): i=0, \dots, k)$ satisfies the desired result. Note that if $h$ is nontrivial, then one of the $N_i$ is greater or equal $2$, and hence $\delta\geq \frac{1}{N}$.
  Since every nilpotent $h\in \lSL_N(\RR)$ is conjugate to a Jordan normal form, the result follows.
\end{proof}

Recall the definition of $\rot(\cdot)$ from \eqref{eq: choice of z()}.

\begin{lemma}\label{lem: contraction long int}
  For every nontrivial $\rho: G\to \SL_N(\RR)$ there exists diagonalizable flow $t \in \R \mapsto b_t\in  \SL_N(\RR)$
  and constants $\delta_1, \delta_2>0$ such that for every $t_1<\delta_1 t_0$, we have the following:
  For every interval $J\subseteq [0,1]$ with $|J| \ge \euler^{-rt_0 + t_1}$ and for every $v\in \RR^N$ we have
  \begin{align}
    \label{eq: main lemma ineq}
    \max_{s\in J} \|b_{t_1}\rho(a_{t_0}\rot(s)\varphi(s)) v\| \asymp \euler^{-\delta_2 t_1}\max_{s\in J}\|\rho(a_{t_0}\varphi(s)) v\|.
  \end{align}
\end{lemma}
\begin{proof}
  Since for every $L<|J|$ and a function $f: J\to \RR$ we have
  \[\max_{s\in J}f(s) = \max_{\substack{J'\subset J\\|J'| = L}}\max_{s\in J'} f(s), \]
  it is sufficient to prove the claim for any subinterval of $J$ with length $\euler^{-rt_0 + t_1}$, hence we may assume that $|J| = \euler^{-rt_0 + t_1}$ and write $J = [s_0, s_0+\euler^{-rt_0 + t_1}]$.
  Set $s=s_0+\varepsilon$ for some $\varepsilon \in \left[0, \euler^{-rt_0 + t_1}\right]$.
  Lemma~\ref{lem: directionalasymp} implies that
  \begin{align*}
    a_{t_0} \varphi(s) = (I+O(\varepsilon+\varepsilon^2\euler^{rt_0})) \exp(\varepsilon\euler^{rt_0}Y_{s_0})a_{t_0}\varphi(s_0).
  \end{align*}
  Hence,
  \begin{align*}
    a_{t_0}\rot(s)\varphi(s)
     & = \rot(s) a_{t_0}\varphi(s)                                                                                                         \\
     & = (I+O(\varepsilon))\rot(s_0)(I+O(\varepsilon+\varepsilon^2\euler^{rt_0})) \exp(\varepsilon\euler^{rt_0}Y_{s_0})a_{t_0}\varphi(s_0) \\
     & = A_{s_0}(\varepsilon)\exp(\varepsilon\euler^{rt_0}Y_{0})\rot(s_0)a_{t_0}\varphi(s_0)
  \end{align*}
  where $A_{s_0}(\varepsilon)\in \SL_d(\R)$ satisfies $A_{s_0}(\varepsilon) = I+O(\varepsilon + \varepsilon^2\euler^{rt_0})$.
  The range for $\varepsilon$ implies that $\varepsilon + \varepsilon^2\euler^{rt_0}\ll \euler^{-rt_0+2t_1} $.

The image of the homomorphism $x\mapsto \rho(\exp(xY_0))$ is a one-parameter unipotent subgroup and, hence, the homomorphism is of the form $x \mapsto \exp(h_0x)$ for some nonzero nilpotent $h_0\in \lSL_N(\RR)$ (explicitly, $h_0 = \mathrm{D}\rho(Y_0)$).
  Let $t \mapsto b(t)$ and $\delta_0>0$ be the diagonalizable flow and the constant from Claim \ref{claim: unipotent homomorphisms bound} for this $h_0$ respectively.
  For every $v\in \RR^N$ we have
  \begin{align}
    \label{eq: LHS of main ineq}
    \max_{s\in J} & \|b(t_1)\rho(a_{t_0}\rot(s)\varphi(s)) v\|                                                                                              \\&= \nonumber
    \max_{\varepsilon\in [0, \euler^{-rt_0 + t_1}]} \|b(t_1)\rho(A_{s_0}(\varepsilon) \exp(\varepsilon\euler^{rt_0}Y_0) \rot(s_0) a_{t_0} \varphi(s_0)) v\| \\
                  & =\nonumber
    \max_{\varepsilon\in [0, \euler^{-rt_0 + t_1}]} \|\rho(A_{s_0}(\varepsilon))^{b(t_1)} b(t_1)\exp(\varepsilon \euler^{rt_0} h_0) v'\|
    ,
  \end{align}
  where $\rho(A_{s_0}(\varepsilon))^{b(t_1)}$ is the conjugation of $\rho(A_{s_0}(\varepsilon))$ by $b(t_1)$ and $$v' = \rho(\rot(s_0) a_{t_0} \varphi(s_0))v.$$
  Since $A_{s_0}(\varepsilon)$ is of distance $O(\euler^{-rt_0+2t_1})$ from the identity matrix, it follows that $\rho(A_{s_0}(\varepsilon))$ is of distance  $O(\euler^{-rt_0+2t_1})$ from the identity matrix as well.
  Since the eigenvalues of $b(t_1), b(t_1)^{-1}$ are bounded by $Rt_1$ for some $R>0$ depending only on $b$, and hence only on $\rho$, it follows that $\rho(A_{s_0}(\varepsilon))^{b(t_1)}$ is of distance
  $\euler^{-rt_0+ (2R+2)t_1}$ of the identity element. Provided that $t_1 < \frac{r}{2R+2}t_0$ we get that $\rho(A_{s_0}(\varepsilon))^{b(t_1)}$ is bounded and hence we may omit this term in Eq. \eqref{eq: LHS of main ineq}, that is, it is proportional to
  \begin{align}\label{eq: LHS of main ineq 1}
    \max_{x\in [0, \euler^{t_1}]} \|b(t_1)\exp(x h_0) v'\|.
  \end{align}
  Now Claim \ref{claim: unipotent homomorphisms bound} implies that
  \begin{align}\label{eq: main lemma ineq 1}
    \max_{x\in [0, \euler^{t_1}]} \|b(t_1)\exp(x h_0) v'\|\asymp \euler^{-\delta t_1}\max_{x\in [0, \euler^{t_1}]} \|\exp(x h_0) v'\|.
  \end{align}
  We have seen that the left-hand side of \eqref{eq: main lemma ineq 1} is asymptotic to the left-hand side of \eqref{eq: main lemma ineq}.
  Moreover, we obtain in a similar manner
  \begin{align*}
    \max_{x\in [0, \euler^{t_1}]} \|\exp(x h_0) v'\|
    \asymp
    \max_{s\in J} \|\rho(\rot(s)a_{t_0}\varphi(s)) v\| \asymp \max_{s\in J} \|\rho(a_{t_0}\varphi(s)) v\|.
  \end{align*}
  using that $s \mapsto \rot(s)$ is bounded, which implies that $s \mapsto \rho(\rot(s))$ is bounded as well. The result thus follows from \eqref{eq: main lemma ineq 1}.
\end{proof}

\subsection{Atypicality of long intervals with short vectors}

Lemma~\ref{lem: contraction long int} together with the quantitative non-divergence results by Kleinbock and Margulis \cite{KleinbockMargulis98} implies the following result, which roughly states that in any representation the set of points $s$ contained in an unexpectedly long interval with a short vector is small (unless a natural obstruction occurs).

\begin{proposition}[Avoiding exceptional intervals]\label{prop: excep int}
  Let $\varphi: [0,1]\to G$ be an analytic curve whose image in $P\backslash G$ is not contained in any weakly constraining pencil or in any partial flag subvariety with respect to $\{a_t\}$.
  Let $\rho: \SL_d \to \SL_N$ be a representation defined over $\Q$ and
  let $\mathcal{U} \subset \Z^N$ be a set of non-zero vectors.
  There exist constants $\delta>0$ and $C_1>0$ depending on $\rho$ with the following property.

  Let $\psi: (0, \infty)\to (0, \infty)$ be a function with $\psi(t) < \delta t$ for all $t >0$.
  Then one of the following is true for all $g \in \SL_d(\R)$, $R>0$, and $t>0$ sufficiently large (depending on $g, R$):
  \begin{enumerate}
    \item Let $\mathcal{VB}_t$ be the union of all intervals $I \subset [0,1]$ with $|I|\geq \euler^{-rt+\psi(t)}$ and with
          \begin{align*}
            I \subset \{s \in [0,1]: \norm{\rho(a_t \varphi(s)g)v} <R\}
          \end{align*}
          for some $v \in \mathcal{U}$.
          Then $|\mathcal{VB}_t|\le C_2\euler^{-\delta\psi(t)}$ for a constant $C_2 = C_2(R, g)$.
    \item There exists a $\rho(\SL_d(\R))$-invariant vector $v \in \mathcal{U}$ with $\norm{v} \leq C_1 R$.
  \end{enumerate}
\end{proposition}

We shall heavily use the following result from linear dynamics (often referred to as 'basic lemma') -- see also Shah's basic lemma in \cite[\S4]{Shah-translates1} as well as \cite[\S4]{ShahYang-translates}.

\begin{theorem}[{Yang \cite{Yang-translates}}]\label{thm: Yang1}
  Let $\varphi: [0,1]\to G$ be an analytic curve whose image in $P\backslash G$ is not contained in any weakly constraining pencil.
  Let $\rho: G \to \SL_N(\R)$ be a representation, let $V^+\subset \R^N$ be the sum of all eigenspaces of eigenvalue $>1$ for $\rho(a_t)$, $t\geq 0$, and write $\pi_+: \R^N \to V^+$ for the projection with respect to the weight decomposition.
  Let $v \in \R^N$ be a nonzero vector such that for all $s \in [0,1]$
  \begin{align}\label{eq: nonexpandedcurve}
    \pi_+(\rho(\varphi(s))v) = 0.
  \end{align}
  Then $\rho(G)v$ is closed. Moreover, if the image $\varphi$ in $P\backslash G$ is not contained in any partial flag subvariety with respect to $\{a_t\}$, then $v$ is $G$-invariant.
\end{theorem}

\begin{proof}
  We merely point out how to combine the proven statements in \cite{Yang-translates}.
  If $\rho(G)v$ were not closed, \cite[Props.~2.7, 6.5]{Yang-translates} would imply that the image of $\varphi$ in $P\backslash G$ would be contained in a weakly constraining pencil yielding a contradiction.
  So $\rho(G)v$ is closed. We now proceed as in \cite[Props.~5.4, 5.6]{Yang-translates}.
  In particular, by \eqref{eq: nonexpandedcurve} there exists for every $s \in [0,1]$ an element $\xi(s) \in G$ with
  \begin{align*}
    \lim_{t \to \infty} \rho(a_t \varphi(s))v = \rho(\xi(s)\varphi(s))v =: w_s
  \end{align*}
  Fix $s \in [0,1]$ and let $F_s<G$ be the stablizer subgroup of $w_s$.
  By Matsushima's criterion, $F_s$ is reductive.
  It is easy to see that $F_s$ contains $\{a_t\}$.

  We claim that we may assume $\xi(s) \in P$.
  Indeed, by definition of $w_s$ there exists for every $t>0$ some $\varepsilon_t$ in a small neighborhood of the identity of $G$ such that $w_s = a_t \xi(s) a_{t}^{-1} \varepsilon_t w_s$.
  In particular, $a_t \xi(s)^{-1} a_{t}^{-1}$ is bounded in $G/F$ which implies that $\xi(s)\in FP$ by a local calculation (indeed, $\xi(s)a_t^{-1}$ is in a neighborhood of the identity coset of $F\backslash G$ for $t$ large enough). This proves the claim.

  Now let $g = \xi(0)\varphi(0)$ and $F = F_0$.
  As $\xi(0) \in P$ and $a_t \varphi(s)v$ converges, $a_t \xi(0)\varphi(s)g^{-1}F$ converges in $G/F$.
  Moreover, $\xi(0)\varphi(s)g^{-1}$ is in a neighborhood of the identity for $s$ close to $0$ and a Lie algebra calculation shows that $\xi(0)\varphi(s)g^{-1} \in PF$.
  Hence, $\varphi(s) \in PFg$ for all $s \in [0,1]$ where $PFg$ is a proper flag subvariety of $P \backslash G$ unless $F =G$ i.e.~$v$ is $G$-invariant.
\end{proof}

\begin{proof}[Proof of Proposition \ref{prop: excep int}]
  Suppose for simplicity first that $\rho$ is irreducible (and nonzero).
  Let $\rot(\cdot)$ be as in \eqref{eq: choice of z()}, let $\delta>0$ to be determined, and let $t \mapsto b_t $ be as in Lemma~\ref{lem: contraction long int}.
  If $\delta$ is sufficiently small, then for any interval $J \subset [0,1]$ with $|J|\geq \euler^{-rt+\psi(t)}$ and any $v \in \R^N$ with
  \begin{align*}
    \max_{s \in J}\norm{\rho(a_t\varphi(s))v} \leq R
  \end{align*}
  we have for some constant $c = c(\rho)>0$
  \begin{align}\label{eq: applycontr}
    \max_{s \in J}\norm{b_{\psi(t)}\rho(\rot(s)a_t\varphi(s))v} \leq c R \euler^{-\delta \psi(t)}.
  \end{align}

  We now invoke the quantitative non-divergence results by Kleinbock and Margulis \cite[Thm.~5.2]{KleinbockMargulis98}. To that end, note first that for any $v \in \R^N$ the map
  \begin{align*}
    s \mapsto \norm{b_{\psi(t)}\rho(a_t \rot(s)\varphi(s))v}
  \end{align*}
  is $(C, \alpha)$-good for some $(C, \alpha)$ depending on $\rho$ and $\varphi$.
  Indeed, we may apply Lemma~\ref{lem: C, alpha good} for $s \mapsto \rho(\rot(s)\varphi(s))$ and $g = b_{\psi(t)}\rho(a_t)$.
  In particular, \cite[Thm.~5.2]{KleinbockMargulis98} implies that for some sufficiently small $\kappa>0$ one of the following options holds for every $t>0$:
  \begin{enumerate}[(a)]
    \item The measure of the set of $s \in [0,1]$ so that
          \begin{align*}
            \norm{b_{\psi(t)}\rho(a_t \rot(s)\varphi(s)g)v} \leq c \euler^{-\delta\psi(t)}R
          \end{align*}
          for some nonzero $v \in \Z^N$
          is $\ll \euler^{-\kappa\delta\psi(t)}$.
    \item There exists a rational subspace $U_t \subset \Q^N$ such that
          \begin{align*}
            \sup_{s \in [0,1]}\norm{b_{\psi(t)}\rho(a_t\rot(s) \varphi(s)g)U_t(\Z)}
            \leq c \euler^{-\frac{\delta}{2}\psi(t)}R.
          \end{align*}
  \end{enumerate}
  We prove that option (b) can only occur in a bounded set of times (depending on $g, R$); this concludes the special case when $\rho$ is irreducible in view of \eqref{eq: applycontr}.
  Suppose option (b) holds for a sequence $t_j \to \infty$.
  Restricting to a subsequence we may assume that the subspaces $U_j=U_{t_j}$ have equal dimension $m$.
  Also, we may assume that if $u_j \in \bigwedge^m \Z^N$ is a pure integral vector corresponding to $U_j(\Z)$ then $\frac{u_j}{\norm{u_j}} \to u$ for some (possibly irrational) pure vector $u \in \bigwedge^m W_i$ of norm one.
  Note that (b) implies
  \begin{align}\label{eq: slowexp}
    \sup_{s \in [0,1]}\norm{\rho(a_{t_j}\varphi(s)g)u_j} \ll \euler^{\star\psi({t_j})}R.
  \end{align}

  Let $\pi_+: W = \bigwedge^m \R^N \to W^+$ be the projection onto the $a_t$-expanded subspace $W^+$.
  By definition, $W^+$ is the direct sum of all $a_t$-eigenspaces with eigenvalue $\euler^{\lambda t}$ for some $\lambda>0$ and the projection is with respect to the eigenspace decomposition.
  Note that there exists $\alpha>0$ (depending on $\rho$) such that for any $w \in W$
  \begin{align*}
    \norm{\rho(a_t)\pi_+(w)} \gg \euler^{\alpha t}\norm{\pi_+(w)}.
  \end{align*}
  In particular, for any $s \in [0,1]$ by discreteness of $W(\Z)$
  \begin{align*}
    \norm{\rho(a_{t_j}\varphi(s)g)u_{j}}
     & \gg \norm{\pi_+(\rho(a_{t_j}\varphi(s)g)u_{j})}
    \gg \euler^{\alpha t_j} \norm{\pi_+(\rho(\varphi(s)g)u_j)}           \\
     & \gg \euler^{\alpha t_j}\norm{u_j}\norm{\pi_+(\rho(\varphi(s)g)u)}
    \gg \euler^{\alpha t_j}\norm{\pi_+(\rho(\varphi(s)g)u)}
  \end{align*}
  which implies together with \eqref{eq: slowexp} that
  \begin{align*}
    \sup_{s \in [0,1]}\norm{\pi_+(\rho(\varphi(s)g)u)} \ll \euler^{-\alpha t_j+ \star \psi(t_j)}R.
  \end{align*}
  For $j \to \infty$ this implies that $\pi_+(\rho(\varphi(s)g)u) = 0$. By Theorem~\ref{thm: Yang1} we obtain that $u$ is $\rho(G)$-invariant.
  But $u$ is pure and hence the corresponding subspace is also $\rho(G)$-invariant. By irreducibility of $\rho$, this is a contradiction, which concludes this special case.

  Now suppose that $\rho$ is general.
  Write $\R^N$ as a direct sum of subspaces
  \begin{align}\label{eq: decomp irred}
    \R^N = \bigoplus_{i=0}^k W_i
  \end{align}
  where $W_0$ is the subspace of $G = \SL_d(\R)$-invariant vectors and where $W_i$ for $i>0$ are irreducible subspaces (all defined over $\Q$).
  This is possible because $\rho$ is defined over $\Q$.
  We may replace $R$ by a constant multiple and assume that the norm $\norm{\cdot}$ is given by $\norm{v} = \max_i \norm{\pi_i(v)}$ where $\pi_i: \R^N \to W_i$ denotes the $G$-invariant projection relative to \eqref{eq: decomp irred}.
  We let $M>0$ be a sufficiently large integer so that $M \pi_i(\Z^N) \subset W_i(\Z)$.

  Let $I \subset [0,1]$ be an interval with $|I|\geq \euler^{-rt+\psi(t)}$ and with
  \begin{align*}
    I \subset \{s \in [0,1]: \norm{\rho(a_t \varphi(s)g)v} <R\}
  \end{align*}
  for some $v \in V(\Z)\setminus\{0\}$.
  Then for every $i=0, \ldots, k$, the interval $I$ is contained in
  \begin{align*}
    \{s \in [0,1]: \norm{\rho(a_t \varphi(s)g)\pi_i(v)}< R\}.
  \end{align*}
  By the already proven special case, we know that for all $i>0$, the union of all connected components of the above sets of length at least $\euler^{-rt + \psi(t)}$ and with $\pi_i(v) \neq 0$ has measure $\ll MR \euler^{-\kappa\delta\psi(t)}$ for some $\delta, \kappa>0$ depending only on~$N$.
  Then either (1) holds or there exists an interval $I$ and a vector $v\in \mathcal{U}$ as above with the additional property that $\pi_i(v) = 0$ for all $i > 0$. But then $v \in W_0(\Z)$ as well as $\norm{v}\leq R$ so (2) holds.
\end{proof}

\section{The Short Intervals Lemma}
\label{sec: short intervals lemma}
For every interval $I = [a-b, a+b]\subset \RR$ denote by $|I| := 2b$ the length and by $\gamma . I = [a-\gamma b, a+\gamma b]$ for all $\gamma>0$ the centered $\gamma$-dilation as in \S\ref{sec: Calpha}.

\begin{lemma}[Short Intervals]\label{lem: martingale interval lemma}
  Fix $0<\beta<1<\gamma$ and a function $f: \NN\to \NN$ such that $\lim_{n\to \infty}\frac{f(n)}{n/\log\log n} = 0$.
  Let $(A_n)_{n=0}^\infty$ be a sequence of sets of intervals in $[-1,1]$ such that for every $I\in A_n$ we have
  $\beta^{n+f(n)} \le |I| \le \beta^{n-f(n)}$ and the intervals $\{\gamma.I: I\in A_n\}$ are disjoint.
  Let
  $B_n = \bigcup_{I\in A_n}I$. Then
  \[\limsup_{N\to \infty}\frac{1}{N}\#\{1\le n\le N: x\in B_n\} \le \frac{1}{\gamma} \text{ a.s. for }x\in [-1,1]\]
\end{lemma}

Before proving the lemma we state an important tool which was introduced by Azuma, \cite[Thm.~1]{azuma1967weighted}. 
The form we use here is not identical to the one given in Azuma's work, and we will provide a reduction from \cite[Thm.~1]{azuma1967weighted} to Lemma \ref{lem: azuma ineq} below.
\begin{definition}
  Let $(M_n)_{n=0}^\infty$ be a sequence of random variables such that $M_n$ is $\Sigma_n$ integrable, where $\Sigma_n$ are a filtration of a probability space $(\Sigma, \mu)$.
  The sequence $(M_n)_{n=0}^\infty$ is called a \emph{martingale} if $\EE(M_n|\Sigma_{n-1}) = M_{n-1}$. It is called \emph{super-martingale} if $\EE(M_n|\Sigma_{n-1}) \le M_{n-1}$.
\end{definition}
\begin{lemma}[Azuma's inequality]\label{lem: azuma ineq}
  For every super-martingale $(M_n)$ and a sequence of positive numbers $c_n$ such that $|M_n-M_{n-1}|\le c_n$ for every $n\ge 1$ one has
  \begin{align}\label{eq: azuma ineq}
    \limsup_{n \to \infty} \frac{M_n}{\sqrt {2 (c_1^2+\ldots+c_n^2)\log n}} \le 1
  \end{align}
  almost surely.
\end{lemma}
\begin{proof}[{Reduction to \cite[Thm.~1]{azuma1967weighted}}:]
  First, whenever $(M_n)_{n=0}^\infty$ is a martingale, \eqref{eq: azuma ineq} follows from \cite[Eq.~(3.1)]{azuma1967weighted} where $x_n = (M_{n}-M_{n-1})/c_n$ and $a_{nk}=c_k$.

Denote the filtration of $(M_n)_{n=0}^\infty$ by $(\Sigma_n)_{n=0}^\infty$.
If $M_n$ is only a super-martingale, we will find a martingale $(M'_n)_{n=0}^\infty$ defined on the same filtration such that $M_n\le M'_n$ for all $n\ge 0$. 
We will then get the desired result by using the lemma already established for $(M_n')_{n=0}^\infty$.

We will construct $(M_n')_{n=0}^\infty$, by setting $M_0' = M_0$ and recursively choosing $M_n'$ by demanding $M_n'-M_{n-1}' \ge M_n-M_{n-1}$ and $\EE(M_n'-M_{n-1}'|\Sigma_{n-1}) = 0$. To this end, we show how to construct for every random variable $X$ with $\EE(X) \le 0$
  another random variable $X'$ with
  \begin{itemize}
    \item
          $\EE(X')=0$,
    \item
          $X'\ge X$ a.s., and
    \item if $X\in [-c,c]$ almost surely then so is $X'$.
  \end{itemize}
  Then we could use this construction to define $M_n'-M_{n-1}'$ from $M_n-M_{n-1}$ conditioned on $\Sigma_{n-1}$.
  If $X$ is the constant random variable, define $X'=0$.
  If $X\le 0$ a.s., define $X' = X-\EE(X)$. Otherwise, denote by $\max(X)$ the minimal value such that $\PP(X \le \max(X)) = 1$. By the assumption, $\max(X) > 0$ and $\max(X) > \EE(X)$.
  Then we choose $X' = \frac{-\EE(X)}{\max(X)-\EE(X)}\max(X) + \frac{\max(X)}{\max(X)-\EE(X)}X$, this is a convex combination between $X$ and $\max(X)$, and it can be seen to satisfy the conditions above.

\end{proof}
\begin{claim}\label{claim: volume of intersection with sparse intervals}
  Let $A$ be a collection of intervals such that $\gamma.I$ are disjoint for $I\in A$. Let $J$ be an interval in $\RR$.
  Then the total length of $J\cap \bigcup_{I\in A}I$ is at most $|J|/\gamma + 2\max_{I\in A}|I|$.
\end{claim}
\begin{proof}
  Let $a = \max_{I\in A}|I|$.
  Without loss of generality assume all $I\in A$ intersect $J$ non-trivially.
  Then
  \[\left| J\cap \bigcup_{I\in A}I \right| \le \sum_{I\in A}|I|.\]
  On the other hand, if we enlarge $J$ by $\gamma a$ on each side it would contain all intervals $\gamma.I$ for $I\in A$.
  Hence $|J| + 2\gamma a\ge \sum_{I\in A}\gamma|I|$. The result follows.
\end{proof}

\begin{proof}[Proof of Lemma \ref{lem: martingale interval lemma}]
  We first perform a few reductions.
  Fix $\delta\in (0,1)$.
  Let $f'(n) := \left\lfloor\delta \frac{n}{\log\log (n+10)}\right\rfloor$; for all but finitely many $n \in \N$ we have $f'(n)\ge f(n)$. Note that $f'$ is monotone increasing and satisfies 
  \begin{align}\label{eq: f' doesnt increasing too much}
    f'(n+1) \le f'(n) + 1,
  \end{align}
  for all $n$. 
  Without changing the correctness of the claim, replace $A_n$ with an empty set whenever $f'(n)/2 < f(n)$.
  Hence it is sufficient to show that if for every $I\in A_n$ we have
  $\beta^{n+f'(n)/2} \le |I| \le \beta^{n-f'(n)/2}$ then
  \begin{align}\label{eq: reformulation 1 of martingale stuff}
    \limsup_{N\to \infty}\frac{1}{N}\#\{1\le n\le N: x\in B_n\} \le \frac{1}{\gamma} + o_\delta(1)\text{ a.s. for }x\in [-1,1].
  \end{align}
  Note that the sequence $(n+f'(n)/2)_{n=1}^\infty$ goes over all integers once, up to a negligible proportion.
  Whenever possible, replace $A_n$ by $A_n' = A_{n'}$ where $n'$ satisfies $n = n' + f(n')/2$ (and by the empty set otherwise).
  With this choice $\beta^{n} \le |I| \le \beta^{n-f'(n)}$ for every $I\in A_n$.

  Fix $a>0$ to be determined later.
  Fix a sequence of positive integers $n_0<n_1<n_2<\dots$ that satisfy $n_{i+1} = n_{i}+f'(n_{i+1})+2a$. Such a sequence exists for every $n_0$ by Eq. \eqref{eq: f' doesnt increasing too much}.
  Denote
  \begin{align*}
    I_m = \{n_{2m-1}, n_{2m-1} + 1, \dots, n_{2m} - 1\}
  \end{align*}
  and let $T_m = n_{2m} + a$.
  Note that for every $n\in I_m$ we have $n\le T_m - a$ and 
  \[n-f'(n) \stackrel{\eqref{eq: f' doesnt increasing too much}}{\ge} n_{2m-1} - f'(n_{2m-1}) = T_{m-1} + a.\]
  Let $J_m = \bigcup_{k=1}^m I_k$.
  It is enough to show that
  \begin{align} \label{eq: what martingale lemma needs}
    \limsup_{m\to \infty} \frac{1}{\#J_m}\sum_{n\in J_m} \mathbbm{1}(x\in B_{n}) & =
    \limsup_{m\to \infty} \frac{1}{\#J_m}\sum_{k=1}^m \sum_{n\in I_k} \mathbbm{1}(x\in B_{n}) \\&\nonumber
    \le \frac{1}{\gamma} + o_{\delta, a}(1)\text{ a.s. for }x.
  \end{align}
  This inequality together with the similar inequality for the shifted sequence $n_i' = n_{i+1}$ proves the bound of Eq.~\eqref{eq: reformulation 1 of martingale stuff} for $N \in (n_i)_{i=0}^\infty$. Since $\frac{n_i-n_{i-1}}{n_i}\xrightarrow{i\to \infty}0$ this sequence is dense enough to imply Eq.~\eqref{eq: reformulation 1 of martingale stuff}.

  Consider a sequence of i.i.d.~uniform $x_n\in [-1,1]$ and some $c_0\in [-1,1]$.
  Define $y_m = c_0 + \sum_{n=0}^{m-1} x_n\beta^{T_n}$ for all $m \in \NN\cup \{\infty\}$.
  Since the law of $y_\infty$ is continuous with positive density on $\left[c_0-\sum_{n=0}^{\infty} \beta^{T_n}, c_0+\sum_{n=0}^{\infty} \beta^{T_n}\right]$, it is sufficient to bound from above
  \[
    \limsup_{m\to \infty} \frac{1}{\#J_m}\sum_{k=1}^{m} \sum_{n\in I_k} \mathbbm{1}(y_\infty\in B_{n}).
  \]
  for all $c_0\in \RR$.
  Since $|y_{\infty}-y_m| < \beta^{T_{m}}\frac{1}{1-\beta}$ we can relate the condition of $y_\infty$ being inside an interval with a similar condition on $y_m$ in the following sense:
  For every $n\in I_m$ any interval $I\in A_n$ is of length $|I|\ge \beta^n$. If $y_\infty$ lies in $I$ then $y_m$ lies in $(1+\varepsilon).I$ for
  $\varepsilon = 2\beta^{a}\frac{1}{1-\beta}>\beta^{T_{m}-n}\frac{1}{1-\beta}$.

  Hence it is sufficient to bound
  \[\limsup_{m\to \infty} \frac{1}{\#J_m}\sum_{k=1}^m \sum_{n\in I_k} \mathbbm{1}(y_k\in B_{n}'), \]
  for $B_n' = \bigcup_{I\in A_n}(1+\varepsilon).I$.

  For every $n\in I_{m+1}$, the conditioned probability that $y_{m+1}\in B_{n}'$ given $y_m$
  is at most $p = \frac{1+\varepsilon}{\gamma} + (1+\varepsilon)\beta^{a}$, by Claim
  \ref{claim: volume of intersection with sparse intervals} applied to
  $J = [y_m-\beta^{T_m}, y_m + \beta^{T_m}]$, $A = \{(1+\varepsilon).I: I\in A_n\}$ and $\gamma/(1+\varepsilon)$ (instead of $\gamma)$.
  Hence the conditional expectation satisfies
  $\EE(\sum_{n\in I_{m+1}} \mathbbm{1}(y_{m+1}\in B_{n}')|y_m)\le p\#I_{m+1}$.
  Thus
  \[M_m = \sum_{k=0}^m \sum_{i\in I_k} \left(\mathbbm{1}(y_\infty\in B_{n}) - p \right), \]
  is a super-martingale, with differences $|M_m-M_{m-1}|\leq \#I_m$.
  Azuma's Inequality in Lemma~\ref{lem: azuma ineq} implies that
  \begin{align}\label{eq: from azuma}
    \limsup_{m\to \infty} \frac{M_m}{\sqrt {2(\sum_{k=0}^m (\#I_k)^2)\log m}} \le 1.
  \end{align}
  We will show that
  \begin{align}\label{eq: resulting ineq}
    \limsup_{m\to \infty} \frac{\sqrt {2(\sum_{k=0}^m (\#I_k)^2)\log m}}{\sum_{k=0}^m\#I_k} = o_\delta(1),
  \end{align}
  as the product of Eqs.~\eqref{eq: from azuma} and \eqref{eq: resulting ineq} implies Eq. \eqref{eq: what martingale lemma needs}.
  Since
  $\sum_{k=0}^m (\#I_k)^2\le \#I_m(\sum_{k=0}^m \#I_k)$, it is sufficient to have
  \[\limsup_{m\to \infty} \frac{\#I_m \log m}{\sum_{k=0}^m\#I_k} = o_\delta(1).\]

  Denote by $M = n_{2m}$.
  Note that $\lim_{m \to \infty}\frac{\sum_{k=0}^m\#I_k}{M} = 1/2$.
  Rewriting $m$ in terms of $M$ we have $m \le \sum_{n=1}^{M} \frac{1}{f'(n)}$. Indeed,
  \[m  = \sum_{k=1}^m 1 = \sum_{k=1}^m \sum_{n\in I_k} \frac{1}{\#I_k} = 
  \sum_{k=1}^m \sum_{n\in I_k} \frac{1}{f'(n_{2k})} \le 
  \sum_{k=1}^m \sum_{n\in I_k} \frac{1}{f'(n)} 
  \le \sum_{n=1}^{M} \frac{1}{f'(n)}.\]
  Using $\#I_m = f'(M)$ we need to show
  \[
    \limsup_{M \to \infty} \frac{f'(M) \log \left(\sum_{n=1}^{M} \frac{1}{f'(n)} \right)}{M} = o_\delta(1).
  \]
  We may replace $f'(M)$ by $\delta M / \log\log M$ and obtain
  \begin{align*}
    \frac{\delta \log \left(\sum_{n=1}^{M} \frac{\log \log n}{\delta n}\right)}{\log\log M} \le
    \frac{\delta \log \left(\log M \log \log M\right) - \delta \log \delta}{\log\log M} + o_M(1) \xrightarrow{M\to \infty} \delta.
  \end{align*}
  Rewriting this in terms of $\#I_m$ we obtain
    \[\limsup_{m\to \infty} \frac{\#I_m \log m}{\sum_{k=0}^m\#I_k} \le 2\delta,\]
  and hence 
  \begin{align}\label{eq: resulting ineq2}
    \limsup_{m\to \infty} \frac{\sqrt {2(\sum_{k=0}^m (\#I_k)^2)\log m}}{\sum_{k=0}^m\#I_k} \le 2\sqrt {\delta}.
  \end{align}

  Altogether, using $\varepsilon= 2\beta^a\frac{1}{1-\beta}$ and $p = (1+\varepsilon)/\gamma + (1+\varepsilon)\beta
    ^a$,
  \begin{align*}
    \limsup_{m\to \infty} \frac{1}{\#J_m}\sum_{k=0}^m \sum_{i\in I_k} \mathbbm{1}(y_k\in B_{n}') & \le
    p + \limsup_{m\to \infty} \left(\frac{M_m}{\#J_m}\right) \stackrel{\eqref{eq: from azuma}\cdot\eqref{eq: resulting ineq2}}{\le} p + 2\sqrt {\delta} \\
    & \le  \frac{1 + \frac{2\beta^{a}}{1-\beta}}{\gamma} + \left(1+\frac{2\beta^{a}}{1-\beta}\right)\beta^{a} + 2\sqrt {\delta} \\
    & =\frac{1}{\gamma}+ O(\beta^a+\sqrt{\delta})
  \end{align*}
  This implies \eqref{eq: what martingale lemma needs} as desired.
\end{proof}

\section{Quantitative non-divergence}
\label{sec: nondivergence}

We write
\begin{align*}
  \height: X \to \R_{>0}, \ g\Gamma \mapsto \inf\{\norm{gv}: 0 \neq v \in \Z^d\}^{-1}
\end{align*}
for the standard height function on $X$. This is a proper and continuous function on $X$.
In this section, we prove the following result.

\begin{theorem}[Quantitative non-divergence]\label{thm: nondiv}
  Let $\varphi: [0,1]\to G$ be an analytic curve so that the image of $\varphi$ in $P \backslash G$ is not contained in any weakly constraining pencil or in any proper flag subvariety with respect to $\{a_t\}$.
  Let $x \in X$, and let $\mu_{s, T}=\frac{1}{T}\int_0^T \delta_{a_t\varphi(s)x}\de t$ for any $T>0$.
  There is a constant $\kappa>0$ so that for almost every $s \in [0,1]$
  \begin{align*}
    \limsup_{T \to \infty}\mu_{s, T}(\{y \in X: \height(y) > \varepsilon^{-1}\}) \ll \varepsilon^\kappa.
  \end{align*}
\end{theorem}

The authors suspect that the assumptions on $\varphi$ can be weakened (see e.g.~the difference in assumptions between \cite[Thms.~1.1, 1.3]{Yang-translates}); the current proof however requires the strong assumptions given above.
The proof uses ingredients from the quantitative non-divergence results of Kleinbock and Margulis \cite{KleinbockMargulis98}, a construction of which we summarize in the following lemma.

\begin{lemma}\label{lem: kleinbockmargulis}
  For any $t>0$ there exists a finite collection $\mathcal{P}_t$ consisting of sets of intervals with the following properties:
  \begin{itemize}
    \item $|\mathcal{P}_t|\leq 1+6+6^2+\ldots+6^{d-1}$.
    \item Any $P \in \mathcal{P}_t$ is a set of disjoint intervals.
    \item For any $P \in \mathcal{P}_t$ and any $J\in P$ there exists a primitive sublattice $\Lambda_J < \Z^d$ so that $J$ is a connected component of
          \begin{align*}
            \{s \in [0,1]: \norm{a_t \varphi(s)g \Lambda_J}\leq 1\}.
          \end{align*}
    \item For any $\varepsilon\in (0,1)$ and any $s\in [0,1]$ with $\height(a_t \varphi(s)x) > \frac{1}{\varepsilon}$ there exists $P \in \mathcal{P}_t$ and $J \in P$ with $\norm{a_t\varphi(s)\Lambda_J}\leq \varepsilon$.
  \end{itemize}
\end{lemma}

Here, $\norm{\Lambda}$ denotes the covolume of a lattice $\Lambda\subset \R^d$ in the subspace it spans. 

\subsection{Proof of Lemma~\ref{lem: kleinbockmargulis}}
As mentioned, a proof of Lemma~\ref{lem: kleinbockmargulis} is (at least implicitly) contained in \cite{KleinbockMargulis98}.
We give here a direct proof for the reader's convenience, which will occupy this subsection.

Write $x = g \Gamma$.
We first construct a finite tree $T_{\varphi, t}$ labeled by pairs $(I, \mathcal{F})$ where $I \subset [0,1]$ is a closed interval and $\mathcal{F}$ is a partial flag of primitive sublattices of $\Z^d$.
This labeling will satisfy that whenever $(I',\mathcal{F}')$ is a descendant of $(I,\mathcal{F})$ we have $I'\subseteq I$ and $\mathcal{F}\subset \mathcal{F}'$ by the construction to follow.
Moreover, any vertex (labeled by) $(I, \mathcal{F})$ in the tree will satisfy
  \begin{itemize}
    \item For all $\Lambda\in \mathcal{F}$
          \begin{align}\label{eq:proptreeA}
            \sup_{s \in I} \norm{a_t \varphi(s)g\Lambda} \leq 1.
          \end{align}
    \item For all $\Lambda\not\in\mathcal{F}$ compatible with $\mathcal{F}$
          \begin{align}\label{eq:proptreeB}
            \sup_{s \in I} \norm{a_t \varphi(s)g\Lambda} \ge 1.
          \end{align}
Here, a primitive lattice $\Lambda \subset \Z^d$ is compatible with a partial flag of primitive lattices $\mathcal{F}$ if $\mathcal{F}\cup\{\Lambda\}$ is a flag.
  \end{itemize}
We point out that we will consider an additional different labeling later on.
  
To construct the root of the tree $T_{\varphi, t}$, let $I_0 = [0,1]$ and $\cF_0$ be a maximal partial flag of primitive sublattices of $\Z^d$ that satisfies that \eqref{eq:proptreeA} holds for all for all $\Lambda\in \cF$ on the interval $I_0$. 
The root $(I_0, \cF_0)$ satisfies \eqref{eq:proptreeB} by its construction.

To construct the rest of the tree $T_{\varphi, t}$, it is sufficient to construct the children of a pair $(I_1, \mathcal{F}_1)$ already in the tree.
For every interval $I \subset I_1$ denote by $\cL_{I}$ the finite collection of primitive lattices $\Lambda\not\in \mathcal{F}_1$ that are compatible with $\mathcal{F}_1$ and for which $I$ is a connected component of $\{s \in I_1: \norm{a_t \varphi(s)g\Lambda}\leq 1\}$. 
  Let $\mathcal{I}$ be the collection of intervals $I \subset I_1$ for which $\cL_{I}$ is nonempty. Note that $\mathcal{I}$ depends on the considered vertex $(I_1,\mathcal{F}_1)$ whose children we are constructing.
  
  \begin{claim}
    The set $\mathcal{I}$ is finite.
  \end{claim}
  \begin{proof}
    For every $\Lambda$ compatible with $\cF_1$, since $s\mapsto \norm{a_t \varphi(s)g\Lambda}^2$ is analytic, we deduce that it is either constant $1$ or $\norm{a_t \varphi(s)g\Lambda}^2 = 1$ has finitely many solutions in $I_1$.
In either case, the set $\{s \in I_1: \norm{a_t \varphi(s)g\Lambda}\leq 1\}$ has only finitely many connected components.
    Hence it is enough to show that there is only finitely many primitive lattices $\Lambda$ compatible with $\cF$ such that $\min_{s\in I_1}\norm{a_t \varphi(s)g\Lambda} \le 1$. 
Any such lattice $\Lambda$ satisfies $\norm{\Lambda} \leq C$ for some constant $C>0$ depending on $\varphi,g,t$. Since there are finitely many (primitive) lattices of covolume at most $C$ we conclude.
  \end{proof}

We aim to cover $\bigcup_{I \in \mathcal{I}} I$ in an `optimal' manner through a subcollection.
Let $\mathcal{I}'$ denote the collection of intervals in $\mathcal{I}$ that are maximal with respect to inclusion. 

  \begin{claim}\label{claim: sup bigger then 1}
    For every primitive lattice $\Lambda$ compatible with $\cF_1$ and $I \in \mathcal{I}'$ we have that 
    \begin{align*}
      \sup_{s\in I}\norm{a_t \varphi(s)g\Lambda} \ge 1. 
    \end{align*}
  \end{claim}
  
  \begin{proof}
    If the claim fails for some interval $I\in \mathcal{I}'$, then $I$ is contained in the set $\{s \in I_1: \norm{a_t \varphi(s)g\Lambda} < 1\}$. 
    If $I \subsetneq I_1$ then this is a contradiction to the maximality of $I$ in $\mathcal{I}$. 
    If $I = I_1$ then we get a contradiction to \eqref{eq:proptreeB} for the vertex $(I_1,\cF_1)$. 
  \end{proof}
  
  Note that $\bigcup_{I\in  \mathcal{I}}I = \bigcup_{I\in  \mathcal{I}'}I$. 
  We take a subcollection $\mathcal{I}'' \subset \mathcal{I}'$ of minimal cardinality so that $\bigcup_{I\in  \mathcal{I}''}I = \bigcup_{I\in  \mathcal{I}}I$. Then necessarily
  \begin{align}\label{eq: sum le 2}
    \sum_{J \in \mathcal{I}''}1_J \leq 2, \qquad 
    \sum_{J \in \mathcal{I}''}1_J(\min(I_1)) \leq 1,\qquad
    \sum_{J \in \mathcal{I}''}1_J(\max(I_1))\leq 1.
  \end{align}
  
  \begin{claim}\label{claim: at most 3}
    Let $\Lambda<\Z^d$ be a lattice compatible with $\cF_1$. 
    Any connected component of $\{s \in I_1: \norm{a_t \varphi(s)g\Lambda}\leq 1\}$ intersects at most $3$ intervals in the collection $\mathcal{I}''$. 
  \end{claim}
  \begin{proof}
    Assume to the contrary that a connected component $I$ of the set $\{s \in I_1: \norm{a_t \varphi(s)g\Lambda}\leq 1\}$ intersects  $4$ different intervals $[e_1,f_1],[e_2,f_2],[e_3,f_3],[e_4,f_4] \in \mathcal{I}''$. 
Since none of these intervals contains another (by minimality of $\mathcal{I}''$), we may assume that $e_1<e_2<e_3<e_4$ and $f_1<f_2<f_3<f_4$ possibly after reordering. 
Since $I \in \mathcal{I}$ intersects the above four intervals non-trivially, so does any interval $[x,y]\in \mathcal{I}'$ containing $I$. 
    Then $x\le f_1$, $y\ge e_4$, and
\begin{align*}
[e_2,f_2],[e_3,f_3] \subset [e_1,f_1]\cup [e_4,f_4] \cup [x,y].
\end{align*}    
Replacing $[e_2,f_2]$ and $[e_3,f_3]$ in $\mathcal{I}''$ with $[x,y]\in \mathcal{I}'$ thus does not change the union over the contained intervals and decreases the size of $\mathcal{I}''$. This contradicts the minimality of the size of $\mathcal{I}''$ and hence the claim follows.
  \end{proof}

The children of the vertex $(I_1, \mathcal{F}_1)$ in the tree are then set to be $(I, \mathcal{F}_1\cup \{\Lambda_I\})$ for $I \in \mathcal{I}''$, where $\Lambda_I$ is any choice of element in $\cL_I$.
Note that any such child satisfies \eqref{eq:proptreeA} (by the definition of $\Lambda_I$) and \eqref{eq:proptreeB} by Claim \ref{claim: sup bigger then 1} using $\mathcal{I}'' \subset \mathcal{I}'$.
This concludes the construction of the tree.
Note that the tree has depth at most $d$ (the exact depth depends on the root only).

We now define the second labeling mentioned earlier.
For any vertex of the tree with label $(I, \mathcal{F})$ and for $(I_1, \mathcal{F}_1)$ its parent we let  $\Lambda_{I, \mathcal{F}}$ be the primitive lattice with $\{\Lambda_{I, \mathcal{F}}\} = \mathcal{F}\setminus \mathcal{F}_1$ (that is, $\Lambda_I$ in the above construction). Furthermore, we let $J(I, \mathcal{F})$ be the connected component of $\{s \in [0,1]: \norm{a_t\varphi(s)g\Lambda_{I, \mathcal{F}}}\leq 1\}$ containing $I$. 
Note that $J(I, \mathcal{F}) \cap I_1 = I$.
Finally, the second label of the given vertex is $J(I,\mathcal{F})$. 

For any two intervals $J,J'$ we write $J \prec J'$ if $\min(J)<\min(J')$ and $\max(J)<\max(J')$.
  The following definition is a summary of the conditions the intervals (given by the second labeling) in the tree obey (see Claim~\ref{claim:weakly ordered} below).
  
  \begin{definition}
    Let $T$ be a finite rooted tree such that every vertex of the tree is labeled by an interval $J(v)$.
Suppose that for every vertex $v$ of the tree its children are ordered as $s_1(v),s_2(v),\dots, s_{d(v)}(v)$. 
We say that $T$ is \emph{weakly ordered} if 
    \begin{enumerate}[(I)]
      \item\label{cond: sorted} for every vertex $v$ we have that $J(s_1(v)) \prec J(s_2(v)) \prec\dots\prec J(s_{d(v)}(v))$. 
      \item\label{cond: contained} For every $i = 2,\dots, d(v) - 1$ we have $J(s_i(v))\subseteq J(v)$
      \item\label{cond: intersects} For every descendant $v'$ of $v$ we have $J(v')\cap J(v)\neq \emptyset$. 
      \item\label{cond: consecutive disjoint} For every $i = 1,\dots, d(v) - 2$ we have that $J(s_i(v)) \cap J(s_{i+2}(v)) = \emptyset$. 
      \item\label{cond: control nbd sons} For every descendant $v'$ of $v$ the interval $J(v')$ intersects at most three intervals out of $J(s_1(v)), J(s_2(v)),\dots,J(s_{d(v)}(v))$.
    \end{enumerate}
  \end{definition}
  \begin{claim}\label{claim:weakly ordered}
The tree $T_{\varphi, t}$ with the second labeling is weakly ordered. 
  \end{claim}
  \begin{proof}
    Let $(I, \mathcal{F})$ be a vertex of the tree. 
    Let $\mathcal{I}, \mathcal{I}', \mathcal{I}''$ be as in the construction of the children of $(I, \mathcal{F})$. 
    The maximality of $\mathcal{I}'$ ensures that for every two children $v_1, v_2$ of $(I, \mathcal{F})$, $J(v_1)$ is not contained in $J(v_2)$ (and vice versa). 
Hence we can sort the children of $(I, \mathcal{F})$ by $(I_1, \mathcal{F}_1), \dots,(I_{d(I, \cF)}, \mathcal{F}_{d(I, \cF)})$ so that Condition \ref{cond: sorted} holds. 
    Equation \eqref{eq: sum le 2} ensures Conditions \ref{cond: contained} and \ref{cond: consecutive disjoint}. 
    Condition \ref{cond: intersects} follows since for every descendant $(I', \cF')$ of $(I, \mathcal{F})$ we have 
    $I'\subseteq J(I', \cF')$ and $I' \subseteq I\subseteq J(I, \mathcal{F})$. 

    To show Condition \ref{cond: control nbd sons}, let $(I', \cF')$ be a descendant of $(I, \mathcal{F})$.
    By Claim \ref{claim: at most 3} we get that $J(I', \cF')$ intersects at most three intervals out of $I_1, \dots, I_{d(I, \cF)}$. 
    This is not enough as Condition \ref{cond: control nbd sons} concerns intersections with $J(I_i, \mathcal{F}_i)$ and not $I_i$. 
    However, since $I_i = J(I_i, \mathcal{F}_i) \cap I$ and $J(I', \cF') \cap I \supseteq I' \neq \emptyset$, we deduce that $J(I', \cF')$ intersects $I_i$ if and only if it intersects $J(I_i, \mathcal{F}_i)$. 
    (This follows from the following observation: if $I_1, I_2, I_3$ are intervals and $I_3\cap I_1, I_2\cap I_1 \neq \emptyset$ then $I_2 \cap I_3 \neq \emptyset$ if and only if $I_1 \cap I_3 \cap I_3 \neq \emptyset$. We apply this to $I_1 = I$, $I_2 = J(I_i, \mathcal{F}_i)$ and $I_3 = J(I', \cF')$.)
    Condition \ref{cond: control nbd sons} follows.
  \end{proof}

\begin{claim}\label{claim:partitiontree}
For every weakly ordered tree $T$ and every $\ell\geq 0$ there is a partition $\mathcal{Q}_\ell$ of the sets of vertices of depth $\ell$ such that $\#\mathcal{Q}_\ell \le 6^\ell$ and for every $Q\in \mathcal{Q}_\ell$ the intervals $(J(v))_{v\in Q}$ are disjoint.
Moreover, for any two vertices $v_1,v_2$ at depth $\ell$ belonging to the same partition element of $\mathcal{Q}_\ell$ its predecessors at depth $\ell' < \ell$ belong to the same partition element of $\mathcal{Q}_{\ell'}$.
  \end{claim}
  \begin{proof}
    The construction of the partition proceeds as follows.
    We set $\mathcal{Q}_0$ to be the trivial partition. 
    Suppose we have constructed the partition $\mathcal{Q}_\ell$. 
    Then define
    \[\mathcal{Q}_{\ell+1} = \left\{\bigcup_{v\in A}\{s_i(v): i \equiv b\mod{6}\}:Q\in \mathcal{Q}_\ell, b=0,\ldots,5\right\}.\]
    That is, for every partition element $Q\in \mathcal{Q}_\ell$ we group the children of $Q$ into $6$ parts according to their index modulo $6$.
This construction clearly obeys the requirement regarding the predecessors.

Assume by contradiction that $v_1, v_2$ are two vertices at depth $\ell$ with $v_1,v_2 \in Q$ for some $Q\in \mathcal{Q}_\ell$ and with $J(v_1)\cap J(v_2)\neq \emptyset$.
    Let $v'$ be the first common ancestor of $v_1, v_2$.
    Let $v_1' = s_{i_1}(v')$ 
    be the child of $v'$ in the direction of $v_1$
    and $v_2' = s_{i_2}(v')$
    be the child of $v'$ in the direction of $v_2$.
    Assume without loss of generality that $i_1 < i_2$.
   By the above predecessor property, $v_1'$ and $v_2'$ are in the same partition element.
   In particular $i_2 \ge i_1 + 6$ by construction of the partition. 
    Condition \ref{cond: contained} implies now that $v_1, v_2$ cannot be the children of $v'$ (in which case $v_1 = v_1'$ and $v_2 = v_2'$).
    Then $J(v_1)$ intersects $J(v_1')$, $J(v_2)$ intersects $J(v_2')$ and $J(v_1)$ and $J(v_2)$ intersect. This means that $J(v_1)\cup J(v_2)$ is an interval and must intersect $J(s_i(v'))$ for all $i_1 \le i\le i_2$. However, there are at least $7$ of those intervals, but $J(v_1)$ and $J(v_2)$ are only allowed to intersect at most three such intervals each. This shows a contradiction. 
  \end{proof}

We apply Claim~\ref{claim:partitiontree} to the tree $T_{\varphi,t}$ with the second labeling and write $\mathcal{P}_{\varphi,t,\ell}$ for the corresponding partitions.
The partition $\mathcal{P}_t$ in Lemma~\ref{lem: kleinbockmargulis} (a partition of the full tree) is then the partition obtained from adding the partitions of all depths together.

  Given a vertex of the tree $T_{\varphi,t}$ with (first) label $(I, \mathcal{F})$ we write
  \begin{align*}
    \mathrm{Bad}_t(I, \mathcal{F}; \varepsilon) = \{s \in I: \norm{a_t \varphi(s)g\Lambda_{I, \mathcal{F}}} \leq \varepsilon\}.
  \end{align*}
  Furthermore, we write for $P \in \mathcal{P}_t$
  \begin{align}
    \mathrm{Bad}_{t, P}(\varepsilon)
     & := \bigcup_{(I, \mathcal{F})\in P} \mathrm{Bad}_t(I, \mathcal{F}; \varepsilon)\nonumber                                                                                                           \\
     & \subset \bigcup_{(I, \mathcal{F})\in P} \{s \in J(I, \mathcal{F}): \norm{a_t \varphi(s)g\Lambda_{I, \mathcal{F}}} \leq \varepsilon\} =: \mathrm{Bad}_{t, P}'(\varepsilon)\label{eq: two bad sets}
  \end{align}
  where the latter union is still disjoint as established above.
  Lastly, we set 
\begin{align*}
\mathrm{Bad}_t(\varepsilon) = \bigcup_{P \in \mathcal{P}_t}\mathrm{Bad}_{t, P}(\varepsilon).
\end{align*}

  Now let $s\in [0,1]$ with $\norm{a_t \varphi(s)v}\leq \varepsilon$ for some $v \in \Z^d$ primitive.
  We claim that $s \in \mathrm{Bad}_t(\varepsilon)$.
  The proof proceeds by induction on the depth of the tree.
  To begin, observe that $s \in I_1$ for some vertex with first label $(I_1, \mathcal{F}_1)$ of level one in the tree (this is by the covering property of $\mathcal{I}''$ in the construction of the tree). 
We are done if $\Z v \in \mathcal{F}_1$ or if $\norm{a_t\varphi(s)g\Lambda}\leq \varepsilon$ for some $\Lambda \in \mathcal{F}_1$.

Otherwise, we claim that there exists a child of $(I_1, \mathcal{F}_1)$ whose interval contains $s$ (in either labeling). 
If $\Z v$ is compatible with $\mathcal{F}_1$, the claim is direct from the construction of the tree (again using the covering property of $\mathcal{I}''$).
If not, let $\Lambda_1 \in \mathcal{F}_1$ be the largest element (with respect to inclusion) not containing $\Z v$.
Then $\Lambda_1$ and $v$ generate a lattice $\Lambda$ with $\norm{a_t\varphi(s)g\Lambda}\leq \varepsilon$ and the same estimate is true for the primitive lattice $\Q \Lambda \cap \Z^d$.
This lattice is clearly compatible with $\mathcal{F}_1$ and hence the claim follows.

We proceed in this way running through vertices of the tree and increasing the depth at each step. At each vertex in this induction the point $s$ is contained in its interval and all intervals of the predecessors of the vertex (again, in either labeling).
The induction stops when we have reached a full flag (or earlier).
  \qed

\subsection{Proof of Theorem~\ref{thm: nondiv}}

  Write $x = g \Gamma$. Throughout the argument we may assume that $\varepsilon>0$ is sufficiently small.
  We begin first with $t \in \N$ and choose an arbitrary enumeration of $\mathcal{P}_t = \{P_{t,1}, \ldots, P_{t, k}\}$ as constructed in Lemma~\ref{lem: kleinbockmargulis} where $k \leq 1 + 6 +  \ldots + 6^{d-1} = k_d$. We may assume that $k=k_d$ by setting $P_{t, j} =\emptyset$ for $k < j \leq k_d$.
  Moreover, we set
  \begin{align*}
    \mathrm{Bad}_{t, j}(\varepsilon) = \bigcup_{J \in P_{t, j}}\{s\in J: \norm{a_t\varphi(s)\Lambda_J} \leq \varepsilon\}.
  \end{align*}
  By Lemma \ref{lem: kleinbockmargulis}, it is enough to show that for every $j$ and for almost every $s$
  \begin{align}\label{eq: nondiv interclaim}
    \limsup_{T \to \infty} \frac{1}{T}\#\{1\leq t \leq T: s \in \mathrm{Bad}_{t, j}(\varepsilon)\} \ll \varepsilon^\star.
  \end{align}

  Applying Proposition~\ref{prop: excep int} to all exterior products of the standard representation and the variety of pure vectors therein, we obtain (as the standard representation is irreducible) that the measure of the union of all intervals in $P_{t, j}$ of length at least $\euler^{-rt+\sqrt{t}}$ is $\ll \euler^{-\star\sqrt{t}}$.
  Let $\mathrm{Bad}_{t, j}'(\varepsilon) \subset \mathrm{Bad}_{t, j}(\varepsilon)$ be the set of points not contained in that union. The Borel Cantelli lemma implies that \eqref{eq: nondiv interclaim} is equivalent to
  \begin{align}\label{eq: nondiv interclaim2}
    \limsup_{T \to \infty} \frac{1}{T}\#\{1\leq t \leq T: s \in \mathrm{Bad}_{t, j}'(\varepsilon)\} \ll \varepsilon^\star
  \end{align}
  for almost every $s$.

  We use Lemma~\ref{lem: martingale interval lemma} to establish \eqref{eq: nondiv interclaim2} after further partitioning.
  By Lemma~\ref{lem: C, alpha good} for any $J$ the number of connected components of $\{s \in J: \norm{a_t\varphi(s)g\Lambda_{J}}\leq \varepsilon\}$ is at most $M$ for some integer $M$ depending only on $\varphi$.
  We may hence partition $\mathrm{Bad}_{t, j}'(\varepsilon)$ into at most $M$ disjoint subsets, denoted for simplicity by $B_{t}^i$, so that for any $J \in P_{t, j}$
  \begin{align*}
    B_{t}^i \cap J \subset \{s \in J: \norm{a_t \varphi(s)g\Lambda_{J}} \leq \varepsilon\}
  \end{align*}
  is a connected component.
  At this point, one would like to invoke Lemma~\ref{lem: martingale interval lemma}, but observe that the lower bound $|B_{t}^i\cap J| \geq \euler^{-rt-\sqrt{t}}$ for $t$ large enough is not necessarily satisfied.
  To remedy that, we replace $B_{t}^i \cap J$ by the connected component of $\{s \in J: \norm{a_t \varphi(s)g\Lambda_{J}} \leq 2\varepsilon\}$ containing it.
By Lemma~\ref{lem: directionalasymp}, there exists $c>0$ (depending only on $\varphi$) so that the interval of width $c\euler^{-rt}$ around any point in $B_{t}^i \cap J$ is contained in the new connected component.
In particular, the new connected component has size $\gg \euler^{-rt}$, and we continue with these intervals.
  Recall that by Lemma~\ref{lem: C, alpha good} the function $s \mapsto \norm{a_t\varphi(s)\Lambda_J}$ is $(C,\alpha)$-good for any $J$ where $C$ and $\alpha$ depend only on $\varphi$.
  Hence, by Lemma~\ref{lem:dilation} the conditions of Lemma~\ref{lem: martingale interval lemma} are satisfied for $\gamma = \frac{1}{C(2\varepsilon)^\alpha}$.
  This yields that for almost every $s$
  \begin{align*}
    \limsup_{T \to \infty} \frac{1}{T}\#\{1\leq t \leq T: s \in B_{t}^i\} \ll \varepsilon^\alpha.
  \end{align*}
  Summing over $i$ the estimates \eqref{eq: nondiv interclaim}, \eqref{eq: nondiv interclaim2} follow and after summing over $1\leq j\leq k_d$ for almost every $s\in [0,1]$
  \begin{align*}
    \limsup_{T \to \infty} \frac{1}{T} \sum_{t=1}^T 1_{\{y\in X: \height(y)>{\varepsilon^{-1}} \}}(a_t\varphi(s)x) \ll \varepsilon^{\alpha}.
  \end{align*}
  For any $\delta\in [0,1]$ and any point $y \in X$ we have $\height(a_\delta y)\asymp \height(y)$ and hence the continuous time result in the theorem follows.
  \qed

\section{Unipotent invariance}
\label{sec: unipotent inv}
The proof we present here for unipotent invariance is a variant of the proof given at \cite[Prop.~3.1]{ChaikaEskinBirkhoffGeneric}, \cite[Prop.~2.2]{shi2020pointwise}. Instead of the estimate of correlation they use, we use martingales directly.

For every locally compact separable metric space $X$, denote by $\Delta(X)$ the set of Borel probability measures on $X$.
Note that there is a metric $d(\cdot, \cdot)$ inducing the weak${}^\ast$-topology on $\Delta(X)$.

\begin{definition}
  For every probability space $(\Omega, \Sigma, \PP)$ and a measurable space $X$, denote by $RV(X;\Sigma)$ the set of $X$ valued $\Sigma$-measurable random variables. 
\end{definition}

\begin{lemma}\label{lem: martingale wierd lemma 2}
  Let $X$ be a locally compact metric space and let $(\Sigma_n)_{n=0}^\infty$ be a filtration of a probability space $(\Omega, \Sigma, \PP)$ by countably generated $\sigma$-algebras.
  Let $(p_n)_{n=0}^\infty$ be a sequence of random points in $X$, where $p_n\in RV(X; \Sigma_n)$.
  Let 
\begin{align*}
M\in RV(\Hom_{cont}(X, \Delta(X)); \Sigma)
\end{align*}  
be a random continuous function $x\mapsto M_x$. Assume that the conditional distribution $Law(p_{n+1}|\Sigma_n)\in RV(\Delta(X), \Sigma_n)\subseteq RV(\Delta(X), \Sigma)$ satisfy that with probability $1$ we have
  $$d(Law(p_{n+1}|\Sigma_n), M_{p_n})\xrightarrow{n\to \infty}0.$$
  Then with probability $1$ every partial limit $\mu$ of $\mu_n = \frac{1}{n}\sum_{k=1}^n \delta_{p_k}\in RV(\Delta(X);\Sigma_n)$ is invariant under convolution with $M$, that is:
  \[\mu = \int_X M_x \de\mu(x).\]
\end{lemma}
\begin{proof}
  To show that almost surely every partial limit $\mu$ of $\mu_n$ satisfies that $\mu = \int_X M_x \de\mu(x)$, we should prove that for a countable dense set of functions $F\subset C_c(X)$ and for every $f\in F$ we have $\int_X M_x(f) \de\mu(x) = \int_Xf\de\mu$. Here $M_x(f) = \int_X f \de M_x$.
  To show that it is sufficient to prove that for every $f\in F$, almost surely
\begin{align*}
\int_X M_x(f) \de\mu_n(x) - \int_X f \de\mu_n\xrightarrow{n\to \infty} 0.
\end{align*}  
  By assumption,
  \begin{align}\label{eq: M close to law}
    \int_X M_x(f) \de\mu_n(x) - \frac{1}{n} \sum_{k=1}^n Law(p_{k+1}|\Sigma_k)(f) \xrightarrow{n\to \infty} 0.
  \end{align}
  On the other hand, define
  \begin{align*}
    B_n := \sum_{k=1}^{n-1} (Law(p_{k+1}|\Sigma_k)(f) - f(p_{k+1}))  \in RV(\RR, \Sigma_{n}).
  \end{align*}
  The process $B_n$ is a martingale, as $B_n - B_{n-1} = Law(p_{n}|\Sigma_{k-1})(f) - f(p_{n})$ has expectation $0$ given $\Sigma_n$ and is bounded by $c = 2\max |f|$. 
  Since its increments are bounded, Azuma's Inequality in Lemma~\ref{lem: azuma ineq} implies that $\frac{B_n}{n}\xrightarrow{n\to \infty}0$.
  This, together with Eq.~\eqref{eq: M close to law} implies the desired result.
\end{proof}

For every $T, s$ denote by $\mu_{s, T} = \frac{1}{T}(t\mapsto a_t \varphi(s)\Lambda)_* m_{[0, T]}$ the pushforward of the uniform measure on $[0, T]$ where $\Lambda\in X_d = \SL_d(\R)/\SL_d(\Z)$ is fixed.
Also, recall the definition of $Y_s \in \Fg^+$ from Lemma~\ref{lem: directionalasymp}.

\begin{theorem}[Unipotent invariance]\label{thm: unipotent inv}
  For almost every $s$ every partial limit of $\mu_{s, T}$ is invariant under the unipotent flow $x\in \R \mapsto \exp(xY_s)$.
\end{theorem}
\begin{proof}
  Fix $0 < \lambda < 1$, and let $x_0, x_1, \ldots\in [0,1]$ be uniform independent random variables. Define the sigma-algebras $\Sigma_n = \sigma(x_0, \ldots, x_n)$ for $n \geq 1$ and $\Sigma = \sigma((x_i)_{i=0}^\infty)$.
  Let $s = (1-\lambda)\sum_{i=0}^\infty \lambda^i x_i\in RV([0,1], \Sigma)$.
  Then $s$ is chosen at random via a continuous distribution with support $[0,1]$, hence it is sufficient to show that every partial limit of $\mu_{s, T}$ is invariant under $x \mapsto \exp(xY_s)$ for almost every $s$ chosen in this way.
  Let $s_n = (1-\lambda)\sum_{i=0}^n \lambda^i x_i\in RV([0,1], \Sigma_n)$, $T_0 = |\log \lambda|/r\in \RR$ and for every $n$ consider the random point
\begin{align*}
p_n = a_{nT_0}\varphi(s_n)\Lambda \in RV(X, \Sigma_n)
\end{align*}  
and the random measures
\begin{align*}
\mu_n &= \frac{1}{n}\sum_{i=0}^{n-1}\delta_{p_i}\in RV(\Delta(X), \Sigma_n)\text{ and }
\mu_n' = \frac{1}{n}\sum_{i=0}^{n-1}\delta_{a_{iT_0}\varphi(s)\Lambda}\in RV(\Delta(X), \Sigma).
\end{align*}
Since $|s-s_i| = O(\lambda^{i+1})$ it follows that $p_i$ and $a_{iT_0}\varphi(s)\Lambda$ are $O(\lambda)$-close (see Lemma~\ref{lem: directionalasymp}).
  Similarly,
  $a_{-t}p_i$ and
  $a_{-t}a_{iT_0}\varphi(s)\Lambda$ are $O(\euler^{-rt}\lambda)$-close when $t < iT_0$.
  Hence $(a_{-t})_*\mu_n$ is $O(\euler^{-rt}\lambda + \frac{1}{\sqrt{n}})$ close to $(a_{-t})_*\mu_n'$ for $t < \sqrt{n}T_0/3$.

  \begin{claim}
    Fix $s\in [0,1]$ and suppose we show that every partial limit of $\mu_{n}$ is $\exp(xY_s)$-invariant for every $x\in \RR$.
    Then every weak-$*$ partial limit of $\mu_{s, T}$ is $\exp(xY_s)$-invariant for every $x\in \RR$.
  \end{claim}
  \begin{proof}
    Let $T_k\xrightarrow{k\to \infty} \infty$ be an infinite sequence of times such that 
    $\mu_{s, T_k}$ converges to a measure $\mu^0$. We need to show that $\mu^0$ is $\exp(xY_s)$-invariant. 
    Since for every $T < T'$ we have $\mu_{s, T} = \mu_{s, T'} + O(\frac{T'-T}{T})$, we may assume that $T_k = T_0n_k$ for some sequence $n_k\to \infty$
    Without loss of generality replace $n_k$ with a subsequence and assume that $\mu_{n_k}$ and $\mu_{n_k}'$ has a weak-$*$ limit as well. 

    Let $\mu' = \lim_{k\to \infty}\mu_{n_k}'$ and $\mu = \lim_{k\to \infty}\mu_{n_k}$.
    By the assumption, $\mu$ is $\exp(xY_s)$-invariant.
    In particular, also $(a_{-t})_*\mu$ is $\exp(xY_s)$-invariant for every $t\ge 0$.
    Moreover, $(a_{-t})_*\mu'$ is $O(\euler^{-rt}\lambda)$ close to $(a_{-t})_*\mu$ for all $t$.
    In addition, it is clear that $\mu'$ is invariant under $(a_{-t})_*$ for $t=T_0$.
    These facts imply that $\mu'$ is invariant under $\exp(xY_s)$ for every $x\in \RR$.
    Since $\mu^0$ is a convolution of $\mu'$ with a uniform distribution of $a_{[0, T_0]}$, for $T = nT_0$, we deduce that $\mu_0$ is $\exp(xY_s)$-invariant.
  \end{proof}
  Hence the proof is reduced to proving the following claim:
  \begin{claim}
    For almost every $s$, every partial limit of $\mu_{n}$ is invariant $x \mapsto \exp(xY_s)$.
  \end{claim}
  To prove the claim we introduce a new operator, and we show that every partial limit of $\mu_{n}$ is invariant under the new operator.
  Let $\nu_s\in RV(\Delta(G), \Sigma)$ denote the following random measure on $G$: choose $y$ uniformly on $[0, \lambda(1-\lambda)]$ and set $\nu_s = Law(a_{T_0} \exp(yY_s))$. Alternatively, if $\nu_s' = Law(\exp(yY_s)) \in RV(\Delta(G^+), \Sigma)$ then $\nu_s = \delta_{a_{T_0}} * \nu_s'$.
  Now, the law $Law(p_{n+1}|\Sigma_n)$ can be computed as follows. We know $s_n$, hence by Lemma~\ref{lem: directionalasymp}
  \begin{align}
    \label{eq: pn+1 law}
    \begin{split}
      p_{n+1}
      & = a_{(n+1)T_0}\varphi(s_{n+1})\Lambda                                   \\
      & = a_{T_0}(I+O(\lambda^{n+1}))\exp((1-\lambda)\lambda x_{n+1}Y_{s_n})p_n \\
      & = (I+O(\lambda^n)) a_{T_0}\exp((1-\lambda)\lambda x_{n+1}Y_{s_n})p_n    \\
      & = (I+O(\lambda^n)) a_{T_0}\exp((1-\lambda)\lambda x_{n+1}Y_{s})p_n.
    \end{split}
  \end{align}
  Here we use that $\lambda$ and $a_{T_0}$ are fixed and hence multiplication by a single $\lambda$ and conjugation with $a_{T_0}$ are absorbed in the $O$-notation.
  Comparing Eq. \eqref{eq: pn+1 law} to 
  \[\nu_s * \delta_{p_n} = Law(a_{T_0}\exp((1-\lambda)\lambda yY_{s})p_n|s, p_n),\]
  we can see that $d(Law(p_{n+1}|\Sigma_n), \nu_s * \delta_{p_n})\xrightarrow{n\to \infty} 0$.
  Now Lemma \ref{lem: martingale wierd lemma 2} shows that every partial limit $\mu$ of $\mu_{n}$ is invariant to convolution with $\nu_s$.
  The theorem is now reduced to the following claim:
  \begin{claim}
    Every measure $\mu$ for which $\nu_s*\mu = \mu$ is invariant under $x \mapsto \exp(xY_s)$.
  \end{claim}
  Inductively, $\mu = \nu_s^{*n}*\mu$. Note that as $\nu_s = \delta_{a_{T_0}} * \nu_s'$ and $\nu_s' = Law(\exp(yY_s))$ for $y\sim Unif([0, \lambda(1-\lambda)])$ we get that $Law(\exp(\kappa yY_s))$ convolution-commutes with $\nu_s'$ for all $\kappa>0$ and
  $$\delta_{a_{T_0}} * Law(\exp(AyY_s)) = Law(\exp(\lambda^{-1}AyY_s))*\delta_{a_{T_0}},$$
  for every $A\in \RR$. Hence
  \begin{align*}
    \mu & = \nu_s^{*n}*\mu = \nu_s^{*(n-1)}*\delta_{a_{T_0}} * Law(\exp(yY_s)) * \mu \\&= \ldots =
    Law(\exp(\lambda^{-n}yY_s)) * \nu_s^{*(n-1)}*\delta_{a_{T_0}} * \mu.
  \end{align*}
  Hence, $\mu$ is the convolution of $Law(\exp(\lambda^{-n}yY_s))$ with some other measure for every $n$.
  This implies that $\mu$ is $\exp(xY_s)$-invariant for all $x\in \RR$.
\end{proof}

\section{Avoiding singular sets and a proof of Theorem~\ref{thm: main}}
\label{sec: avoid sets and proof}

In this section, we combine the linearization technique developed by Dani and Margulis \cite{DaniMargulis93} with the methods from the previous sections to establish that for almost every point on the curve the trajectory under the diagonal flow does not accumulate on a singular set.

\subsection{The linearization technique}

Consider the one-parameter unipotent subgroup $W=\{\exp(xY_0): x \in \R\}< G^+$.
By Theorem~\ref{thm: unipotent inv}, for almost every $s$ any limit of the measures
\begin{align*}
  \nu_{s, T} = \frac{1}{T}\int_0^T \delta_{\rot(s)a_t\varphi(s)x} \de t = \rot(s)_\ast \mu_{s, T}
\end{align*}
is $W$-invariant. Here, $x \in X$ will remain fixed.

Let $\mathcal{H}$ be the collection of all closed connected subgroups of $G= \SL_d(\R)$ such that $H \cap \Gamma < H$ is a lattice and such that the subgroup of $H$ generated by all unipotent one-parameter subgroups acts ergodically on $H\Gamma/\Gamma$.
It is easy to see from the Borel density theorem that $\mathcal{H}$ is countable (in fact, this would also be true if $\Gamma$ were non-arithmetic \cite[Prop.~2.1]{DaniMargulis93}).
We define for $W$ as above and $H \in \mathcal{H}$
\begin{align*}
  N(W, H) & = \{g \in G: W \subset gHg^{-1}\},                   \\
  S(W, H) & = \bigcup_{\mathcal{H}\ni H' \subsetneq H} N(W, H').
\end{align*}
The following is a consequence of Ratner's measure classification \cite{ratner91-measure} (see also \cite[Thm.~2.2]{mozesshah}):

\begin{theorem}\label{thm: ratner}
  Let $\mu$ be a $W$-invariant probability measure on $X = G/\Gamma$ with $\mu(N(W, H)\Gamma) =0$ for all proper subgroups $H \in \mathcal{H}$ of $G$.
  Then $\mu$ is the $G$-invariant probability measure on $X$.
\end{theorem}

By Theorems~\ref{thm: nondiv} and \ref{thm: unipotent inv}, the following proposition thus implies Theorem~\ref{thm: main}.

\begin{proposition}\label{prop: linearization}
  For almost every $s\in [0,1]$ and any weak${}^\ast$-limit $\nu$ of the measures $\nu_{s, T}$ (as $T\to \infty$) we have $\nu(N(W,H)\Gamma/\Gamma)= 0$.
\end{proposition}

As mentioned we apply the linearization technique to prove Proposition~\ref{prop: linearization}.
For any $H \in \mathcal{H}$ we consider the $G$-space (via the adjoint representation)
\begin{align*}
  V_H = \bigwedge^{\dim(H)}\Fg
\end{align*}
and a unit vector $v_H \in \bigwedge^{\dim(H)}\Fh$. Whenever not implicitly clear, we write $\rho_H: G \to \SL(V_H)$ for the representation and $\eta_H: G \to V_H, \ g \mapsto g.v_H$ for the orbit map at $v_H$.
Observe that $N(W, H) = \eta_H^{-1}(A_H)$ where
\begin{align*}
  A_H = \{v \in V_H: Y_0 \wedge v_H = 0 \}.
\end{align*}
For any $R >0$ we write $B_R = \{v \in V_H: \norm{v} \leq R\}$.
Finally, we define
\begin{align*}
  S_{H, R} = \{g \in G: \exists \gamma \in \Gamma \text{ with } \eta_H(g\gamma) \neq \pm \eta_H(g) \text{ and }  \eta_H(g\gamma), \eta_H(g) \in B_R\cap A_H\}.
\end{align*}

The following proposition is a consequence of the fact that
\begin{align*}
  G/{\Stab_{\Gamma}(v_H)} \to G/\Gamma \times V_H, \ g \mapsto (g\Gamma, \eta_H(g))
\end{align*}
is proper.

\begin{proposition}[{\cite[Prop.~3.2]{mozesshah}}]\label{prop: lin-avoiding special sets}
  Let $R>0$. Then $S_{H, R}$ is a subset of $S(W, H)$ with $S_{H, R}\Gamma \subset X$ closed.
  For any compact subset $K \subset X \setminus S_{H, R}\Gamma$ there exists a neighborhood $\mathcal{O}$ of $B_R \cap A_H \subset V_H$ such that any coset $\eta_H^{-1}(\mathcal{O})\Gamma \cap K$ has a unique representative in $\eta_H^{-1}(\mathcal{O})$.
\end{proposition}

\begin{proposition}\label{prop: linearization-time in tubes}
  Let $r >0$, $H \in \mathcal{H}$, and let $\varepsilon>0$. Then there exists a closed subset $\mathcal{S} \subset S(W, H)$ so that for any compact subset $K \subset X \setminus \mathcal{S}\Gamma$ there exists an open neighborhood $\mathcal{O}'$ of $B_r\cap A_H$ with the following property.
  For almost every $s \in [0,1]$ and $T$ sufficiently large
  \begin{align*}
    \nu_{s, T}(K \cap \pi(\eta_H^{-1}(\mathcal{O}'))) \leq \varepsilon.
  \end{align*}
\end{proposition}

\begin{proof}
  In the following, we first take times $t \in \N$.
  Let $\delta\in (0,1)$ (to be chosen polynomial in $\varepsilon$ later) and $R = r/\delta$. Set $\mathcal{S} = S_{H, R}$, let $K \subset X \setminus \mathcal{S}\Gamma$ be compact, and choose $\mathcal{O}$ as in Proposition~\ref{prop: lin-avoiding special sets}. Let $\beta>0$ so that $\mathcal{O}$ contains
  \begin{align*}
    \{v \in V_H: \norm{v}< R+\beta, \norm{Y_0 \wedge v}< \beta\}.
  \end{align*}
  By replacing $\mathcal{O}$ we may assume that it is equal to the above displayed set.
  We will show that the open neighborhood
  \begin{align*}
    \mathcal{O}'
    = \{v \in V_H: \norm{v}< \delta(R+\beta), \norm{Y_0 \wedge v}< \delta\beta\}= \delta \mathcal{O}.
  \end{align*}
  of $B_r \cap A_H$ satisfies the requirements of the proposition.

  Consider the set of points
  \begin{align*}
    J_t= \{s \in [0,1]: \rot(s)a_t \varphi(s)x \in K \cap \pi(\eta_H^{-1}(\mathcal{O}))\}
  \end{align*}
  Note that for any $s \in J_t$ there exists a unique vector $v_s \in \eta_H(x)$ (up to signs) such that $\rot(s)a_t \varphi(s).v_s \in \mathcal{O}$.
  We let $I_t(s)$ be the largest interval containing $s$ such that $\rot(s')a_t \varphi(s').v_s \in \mathcal{O}$ for all $s' \in I_t(s)$.
  By construction, these intervals are either equal or disjoint.

  We will say that $s$ is special at time $t\in \N$ if $s \in J_t$ and $I_t(s)$ has length at least $\euler^{-rt+ \sqrt{t}}$ and is special if it is special for infinitely many times $t\in \N$. We claim that the set of special points has zero measure.
  As $s \mapsto \rot(s)$ is bounded, there exists a constant $C_1>0$ so that for any $t>0$
  \begin{align*}
    I_t(s) & \subset \{s \in [0,1]: \norm{\rot(s)a_t\varphi(s).v_s}\leq R+\beta\} \\
           & \subset \{s \in [0,1]: \norm{a_t\varphi(s).v_s}\leq C_1(R+\beta)\}.
  \end{align*}
  Also, note that no nontrivial pure vector in $V_H$ is $G$-invariant as $\Fg$ is simple.
  Hence, Proposition~\ref{prop: excep int} applied to the variety of pure vectors in $V_H$ implies that the measure of the set of special points at time $t$ is $\ll \euler^{-\star\sqrt{t}}$ where the implicit constant depends on $R$ (and hence on $\delta$) and $x$. The claim thus follows from the Borel-Cantelli lemma.

  Let $I$ be one of the non-special intervals $I_t(s)$.
  Lemma~\ref{lem: C, alpha good} shows that the functions $s' \mapsto \norm{a_t \varphi(s').v_s}$ and $s' \mapsto \norm{Y_0 \wedge a_t \varphi(s').v_s}$ are $(C, \alpha)$-good where $C$ and $\alpha$ only depend on $\dim(H)$ and $\varphi$.
  In particular,
  \begin{align*}
    f_s: s' \mapsto \max\{(R+\beta)^{-1}\norm{a_t \varphi(s').v_s}, \beta^{-1}\norm{Y_0 \wedge a_t \varphi(s').v_s}\}
  \end{align*}
  is $(C, \alpha)$-good.
  Moreover, any set of the form $\{s':|f_s(s')|< \kappa\}$ has at most $M^2$ connected components where $M$ is as in Lemma~\ref{lem: C, alpha good} for the representation $\rho_H$.
  We define for each $k \leq M^2$ a collection of intervals $\mathcal{I}_{t, k}$ by picking for each interval $I_t(s)$ the $k$-th connected component of $\{s' \in I_t(s):|f_s(s')|< \delta\}$ (if it exists).
  By Lemma~\ref{lem:dilation} we have that for any $I' \in \mathcal{I}_{t, k}$ the dilation $\frac{1}{C\delta^\alpha}.I'$ is contained in one of the intervals $I_t(s)$.
Note that $\mathcal{I}_{t, k}$ satisfies the assumptions of Lemma~\ref{lem: martingale interval lemma} using additionally that the intervals $I_t(s)$ are either equal or disjoint.
  We let $B_{t, k}$ be the union of all intervals in $\mathcal{I}_{t, k}$. 
  By Lemma~\ref{lem: martingale interval lemma}, we have for almost every $s \in [0,1]$
  \begin{align*}
    \limsup_{T \to \infty} \frac{1}{T} \#\{1\leq t \leq T: s \in B_{t, k}\} \ll \delta^\alpha.
  \end{align*}
  Summing over $k$, for almost every $s \in [0,1]$
  \begin{align*}
    \limsup_{T \to \infty} \frac{1}{T} \#\{1 \leq t \leq T: \rot(s)a_t \varphi(s).x \in K \cap \pi(\eta_H^{-1}(\mathcal{O}'))\}\ll \delta^\alpha,
  \end{align*}
  which concludes the discrete-time case of the lemma for a suitable choice of $\delta = \delta(\varepsilon)$.
  To obtain the lemma for continuous time, it suffices to observe that by Lipschitz continuity there exists a constant $C_2\geq 1$ (depending on $\rot(\cdot)$) so that $\rot(s)a_{\lfloor t \rfloor}\varphi(s) \in \mathcal{O}'$ implies $\rot(s)a_{t}\varphi(s) \in C_2\mathcal{O}'$.
\end{proof}

\begin{proof}[Proof of Proposition~\ref{prop: linearization}]
  For any $H \in \mathcal{H}$ write $\mathrm{Bad}_H \subset [0,1]$ for the set of points $s \in [0,1]$ for which the sequence of measures $\nu_{s, T}$ has a weak${}^\ast$-limit $\nu$ satisfying $\nu(N(W, H)\Gamma/\Gamma) > 0$.
  For any compact subset $K' \subset N(W, H)$ and any $\kappa > 0$ the sets $\mathrm{Bad}_H(K', \kappa) \subset \mathrm{Bad}_H$ are similarly defined with the inequality $\nu(K'\Gamma/\Gamma) > \kappa$.
  These are Borel measurable subsets.
  Note that $\mathrm{Bad}_H = \bigcup_{K',\kappa}\mathrm{Bad}_H(K', \kappa)$.
  Let $H \in \mathcal{H}$ be minimal (with respect to inclusion) so that $\mathrm{Bad}_H$ has a positive measure.
  In particular, $\mathrm{Bad}_{H'}$ is a nullset for any $\mathcal{H}\ni H'\subsetneq H$ and for almost every $s \in [0,1]$ any limit $\nu$ of the measure $\nu_{s, T}$ satisfies $\nu(S(W, H)\Gamma) = 0$.
  We suppose by contradiction that $H \neq G$.
  Choose $K' \subset N(W, H)\setminus S(W, H)$ compact and $\kappa>0$ such that $\mathrm{Bad}_H(K', \kappa)$ has positive measure.
  Let $r >0$ be such that $\eta_H(K') \subset B_r$.
  Let $\mathcal{S}$ be as in Proposition~\ref{prop: linearization-time in tubes} for $\varepsilon= \frac{\kappa}{2}$ and choose a compact neighborhood $K \subset X \setminus \mathcal{S} \Gamma$ of $K' \Gamma$.
  By Proposition~\ref{prop: linearization-time in tubes} there exists an open neighborhood $\mathcal{O}$ of $\eta_H(K')$ so that for almost every $s$ and $T$ sufficiently large
  \begin{align*}
    \nu_{s, T}(K \cap \eta_H^{-1}(\mathcal{O})\Gamma) \leq \frac{\kappa}{2}.
  \end{align*}
  For $T \to \infty$ this implies that any weak${}^\star$-limit $\nu$ of $\nu_{s, T}$ satisfies $\nu(K'\Gamma) \leq \frac{\kappa}{2}$.
  This is a contradiction to $\mathrm{Bad}(K', \kappa)$ having a positive measure and the proposition follows.
\end{proof}

\section{Proof of the applications}\label{sec: applications}

\subsection{Proof of Theorem~\ref{thm: dirichletimprovable}}
Let $X = \SL_{d+1}(\R)/\SL_{d+1}(\Z)$.
Write $\norm{\cdot}_\infty$ for the supremum norm on $\R^d$ and set for $\mu \in (0,1]$
\begin{align*}
  X(\mu) = \{\Lambda\in X: \exists\lambda \in \Lambda\setminus\{0\} \text{ with }\norm{\lambda}_{\infty} \leq \mu \}
\end{align*}
Note that $X(1) = X$ and that $X(\mu)$ for $\mu<1$ is a subset of (Haar) measure in $(0,1)$.
We write $f(\mu)$ for that measure.

Define $a_t = \diag(\euler^{dt}, \euler^{-t}, \ldots, \euler^{-t})$ for $t \in \R$ and for $\xi \in \R^d$
\begin{align*}
  u(\xi) = \begin{pmatrix}
             1 & \xi_1 & \cdots & \xi_d \\
               & 1     &        &       \\
               &       & \ddots &       \\
               &       &        & 1
           \end{pmatrix}\in \SL_{d+1}(\R).
\end{align*}
By the Dani correspondence, \eqref{eq: Dirichletimprovable} for $\xi$ and $N$ has a non-zero solution if and only if $a_{\log(N)}u(\xi)\Z^{d+1} \in X(\mu)$.
As $X(\mu)$ is Jordan-measurable (with respect to the Haar measure), Theorem~\ref{thm: main} implies that for almost every $s$ and for every $\mu>0$
\begin{align}\label{eq: logscaleint}
  \begin{split}
    \frac{1}{\log(N)} &\int_1^N \frac{1_{X(\mu)}(a_{\log(n)} u(\phi(s))\Z^{d+1})}{n} \de n \\
    \qquad&= \frac{1}{\log(N)} \sum_{n=1}^N\frac{1_{X(\mu)}(a_{\log(n)} u(\phi(s))\Z^{d+1})}{n} + o(1)
  \end{split}
\end{align}
converges to $f(\mu)$. This together with Dani correspondence implies Theorem~\ref{thm: dirichletimprovable}.

\subsection{Proof of Theorem~\ref{thm: algapprox}}
Let $s\in \RR$.
Define
\[C_s = \left\{(c_i)_{i=1}^d\in \RR^d: |c_i|\le 1 \text{ for }i=1, \dots, d, \left| \sum_{i=1}^d c_is^i\right| \le 1 \right\}\subset \RR^{d}, \]
\[M_{s, \mu} = \left\{\left(\eta,(c_i)_{i=1}^d\right) \in \RR \times \RR^d: (c_i)_{i=1}^d \in C_s,~\mu \left|\sum_{i=1}^d ic_is^{i-1} \right| \ge |\eta|\right\}, \]
\[f(s, \mu) = m_{X_{d+1}}\left( \left\{ \Lambda\in X: \Lambda\cap M_{s,\mu} \neq \{0\} \right\} \right).\]
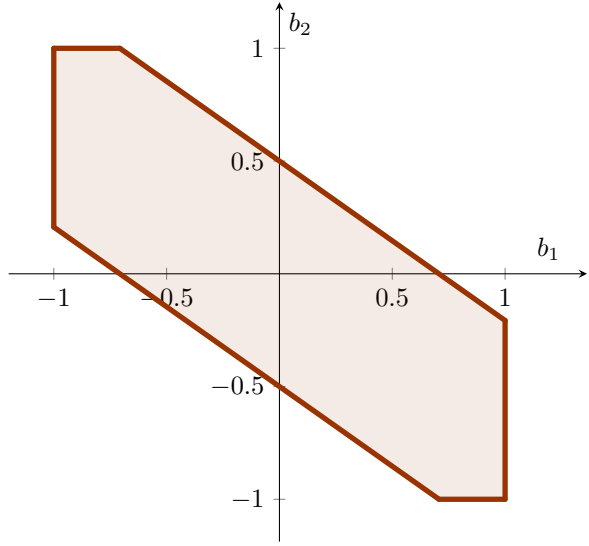
\begin{figure}
  \definecolor{zzttqq}{rgb}{0.6,0.2,0.}
  \begin{tikzpicture}[line cap=round, line join=round,>=triangle 45, x=3.0cm, y=3.0cm]
    \begin{axis}[
        x=3.0cm, y=3.0cm,
        axis lines=middle,
        xmin=-1.1984562985247267,
        xmax=1.3629061714260087,
        ymin=-1.1862653951832431,
        ymax=1.204339576770771,
        xtick={-1.0,-0.5,...,1.0},
        ytick={-1.0,-0.5,...,1.0},]

      \clip(-1.1984562985247267,-1.1862653951832431) rectangle (1.3629061714260087,1.204339576770771);
      \fill[line width=2.pt, color=zzttqq, fill=zzttqq, fill opacity=0.10000000149011612] (-1.,1.) -- (-0.7071067811865478,1.) -- (1.,-0.20710678118654752) -- (1.,-1.) -- (0.7071067811865478,-1.) -- (-1.,0.20710678118654752) -- cycle;
      \draw [line width=2.pt, color=zzttqq] (-1.,1.)-- (-0.7071067811865478,1.);
      \draw [line width=2.pt, color=zzttqq] (-0.7071067811865478,1.)-- (1.,-0.20710678118654752);
      \draw [line width=2.pt, color=zzttqq] (1.,-0.20710678118654752)-- (1.,-1.);
      \draw [line width=2.pt, color=zzttqq] (1.,-1.)-- (0.7071067811865478,-1.);
      \draw [line width=2.pt, color=zzttqq] (0.7071067811865478,-1.)-- (-1.,0.20710678118654752);
      \draw [line width=2.pt, color=zzttqq] (-1.,0.20710678118654752)-- (-1.,1.);
      \draw (0,1.2) node[anchor=north west] {${b_2}$};
      \draw (1.1,0.2) node[anchor=north west] {${b_1}$};
    \end{axis}
  \end{tikzpicture}
  \caption{The set $C_s$ for $d=2, s=\sqrt 2$.}
  \label{fig: Cs}
\end{figure}
\begin{figure}
  \includegraphics[scale=0.5]{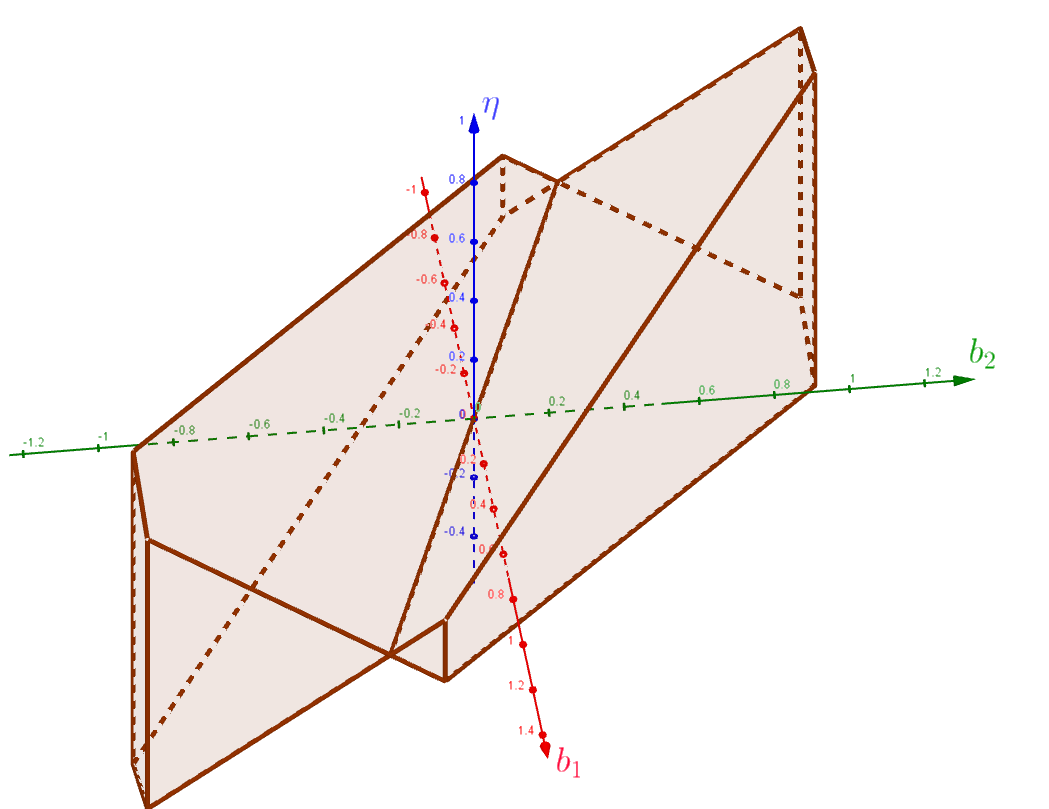}
  \caption{The set $M_{s, \mu}$ for $d=2, s=\sqrt 2, \mu=1/4$.}
  \label{fig: Msmu}
\end{figure}
We will prove Theorem \ref{thm: algapprox} with this function $f$.
The sets $C_s$ and $M_{s, \mu}$ are illustrated respectively in Figures \ref{fig: Cs} and \ref{fig: Msmu}.

Let $a_t$ be as in \eqref{eq: a_t} for $n=d, m=1$, and set for any $s\in \R$ $$\varphi(s) = \begin{pmatrix}
    1 & s & \cdots & s^d \\
      & 1 &        &     \\
      &   & \ddots &     \\
      &   &        & 1
  \end{pmatrix}.$$
Let $\Lambda_{s, t}:= a_t\varphi(s) \ZZ^{d+1}$.
For every $\varepsilon>0$ denote by $M_{s, \mu}^\varepsilon$ the $\varepsilon$-neighborhoods of $M_{s, \mu}$, and $M_{s, \mu}^{-\varepsilon}$ the set of $v\in M_{s, \mu}$ such that the $\varepsilon$-ball around $v$ is contained in $M_{s, \mu}$.

\begin{claim}\label{claim: algebraic approximation is dynamics}
  For every $s\in \RR$, $\mu>0$, $\varepsilon>0$ and $t>t_0(s, \mu, \varepsilon)$ the following occur:
  \begin{enumerate}[\normalfont (\roman*)]
    \item\label{case:1} Let $s'\in \RR$ be a real number with $|s'-s|<\mu \euler^{-t(1+1/d)}$ and $\sum_{i=0}^d b_i(s')^i=0$ for some $(b_i)_{i=0}^d \in \ZZ^{d+1}$ with $|b_i|<\euler^{t/d}$ for all $i=0, \ldots, d$. Then the vector $a_t\varphi(s)((b_i)_{i=0}^d)$ lies in $\Lambda_{s, t}\cap M_{s, \mu}^\varepsilon$.
    \item\label{case:2} Let $(\eta,(c_i)_{i=1}^d)\in \Lambda_{s, t}\cap M_{s, \mu}^{-\varepsilon}$. Then the vector $$(b_i)_{i=0}^d =(a_t\varphi(s))^{-1}(\eta,(c_i)_{i=1}^d)\in \Z^{d+1}$$ satisfies that $|b_i|<\euler^{t/d}$ for all $i=0, \ldots, d$ and $\sum_{i = 0}^d b_i x^i$ has a real root $s'$ with $|s'-s|<\mu \euler^{-t(1+1/d)}$.
  \end{enumerate}
\end{claim}
The proof of Claim \ref{claim: algebraic approximation is dynamics} is a straightforward but lengthy computation, and we leave it to the end of the subsection. We are now ready to prove Theorem \ref{thm: algapprox}.
\begin{proof}[Proof of upper bound in Theorem \ref{thm: algapprox}]
  For an upper bound of $\overline{\mathcal{A}}_s(\mu)$, as in \eqref{eq: logscaleint},
  it is sufficient to bound \[(*)_{T} = \frac{1}{T}\int_{0}^T \mathbbm{1}(\text{\eqref{eq: algapprox} holds for }N'=\euler^{t/d})\de t.\]
  By Claim \ref{claim: algebraic approximation is dynamics},
  \[
    (*)_{T}\le \frac{1}{T}\int_{0}^T \mathbbm{1}(\Lambda_{s, t}\cap M_{s, \mu}^\varepsilon\neq \{0\}) \de t + \frac{t_0(s, \mu, \varepsilon)}{T}
  \]
  Apply Theorem~\ref{thm: main} to $\varphi, \ZZ^{d+1}$, $a_t$. This is possible as the image of $\varphi$ is not contained in any affine hyperplane of the expanding horosphere of $a_t$. This implies that $\lim_{T\to \infty}(*)_{T} \le m_X(\Lambda \in X: \Lambda \cap M_{s, \mu}^\varepsilon \neq \{0\})$ almost surely.
  Since $\bigcap_{\varepsilon>0} M_{s, \mu}^\varepsilon = M_{s, \mu}$ it follows that \[m_X(\Lambda \in X: \Lambda \cap M_{s, \mu}^\varepsilon \neq \{0\})\xrightarrow{\varepsilon\to 0} f(s, \mu).\]
\end{proof}
The lower bound is slightly more complicated, as we need to guarantee that the algebraic approximations are of degree $d$ and not less.
One possible argument for the irreducibility replaces the space $X_{d+1}$ with the congruence cover $\SL_{d+1}(\RR)/(\ker(\SL_{d+1}(\ZZ) \to \SL_{d+1}(\Z/M\Z)))$
where $M=p_1 p_2\cdots p_k$ is a product of distinct primes and obtaining that the resulting polynomials are random independent nonzero polynomials modulo $p_i$.
Since random polynomials $b_0 +b_1x+\dots + b_d x^d\in \FF_p[x]$ are irreducible with probability $\sim\frac{1}{d}$, we get that the integral polynomial is irreducible with high probability.
We choose a more straightforward proof, relying on the fact that if $s'$ is an algebraic number of degree $d' < d$ then it is too good an approximation with this degree, and hence occurs with $0$ frequency. To control the size of the minimal polynomial of $s'$ provided that it satisfies a degree $d$ polynomial with certain coefficients, we introduce the following lemma:
\begin{lemma}\label{lem: div-size}
  Let $Q(x)|P(x)$ integer polynomials of degree $d'<d$.
  Then $\height{(Q)}\le c(d', d)\height(P)$, where $c(d', d)$ is a constant depending only on $d, d'$.
\end{lemma}
\begin{proof}
  For every polynomial $P\in \RR[x]$ denote by ${\rm maxcoef}(P)$ the maximal absolute value of a coefficient of $P$.
  It is sufficient to show that for every $k, l \ge 0$ there exists a constant $C(k, l)>0$ such that for every $Q_1,Q_2\in \RR[x]$ with
  $\deg(Q_1)=k, \deg(Q_2) = l$ we have
  \[{\rm maxcoef}(Q_1)\,{\rm maxcoef}(Q_2)\le C(k, l){\rm maxcoef}(Q_1Q_2).\]
  This holds with
  \[C(k, l) = \left(
    \min_{\substack{{\rm maxcoef}(Q_1) = 1\\{\rm maxcoef}(Q_2)=1}}{\rm maxcoef}(Q_1Q_2)
    \right)^{-1}.
  \]
\end{proof}

\begin{proof}[Proof of lower bound in Theorem \ref{thm: algapprox}]
  Fix $s\in \RR$. We say that $t$ is \emph{almost $d$-algebraic $\mu$-approximable} if it satisfies \eqref{eq: algapprox} for $N=\euler^{t/d}, \mu$ with $s'$ an algebraic number which is a root of a polynomial $P(x) = b_0+b_1x+b_2x^2+\dots+b_dx^d$, which is not necessarily irreducible, with $|b_i|<\euler^{t/d}$.
  If $s'$ is not of degree $d$, there is an integer irreducible polynomial $Q(x)|P(x)$ of degree $d'<d$ such that $Q(s') = 0$.
  To bound the maximal coefficient of $Q$, we use the Lemma~\ref{lem: div-size}.
  Thus, if $t$ is almost $d$-algebraic $\mu$-approximable,
  with an approximation $s'$ of degree $d'<d$, then $s'$ is the root of an integral irreducible polynomial $Q(x)$ of degree $d'$
  and coefficients less than $c(d', d)\euler^{t/d}$ in absolute value.
  By the upper bound of Theorem \ref{thm: algapprox}, we get that for almost every $s$, approximations of $s$ with $s'$ algebraic of degree $d'<d$ with $\|s'\|<\euler^{t/d}c(d', d)$, then for proportion $1$ of $t$, such $s'$ has $s-s' = \omega((\euler^{t/d}c(d', d))^{-d'-1})$.
  Since this approximation is too bad to make $t$ almost $d$-algebraic $\mu$-approximable, it follows that it is enough to bound from below
  \[(**)_{T} = \frac{1}{T}\int_{0}^T \mathbbm{1}(t \text{ is almost $d$-algebraic $\mu$-approximable}) \de t.\]
  As in the proof of the upper bound,
  \[
    (**)_{T}\ge \frac{1}{T}\int_{0}^T \mathbbm{1}(\Lambda_{s, t}\cap M_{s, \mu}^{-\varepsilon}\neq \{0\}) \de t \xrightarrow{T\to \infty}m_X(\Lambda \in X: \Lambda \cap M_{s, \mu}^{-\varepsilon} \neq \{0\}),
  \]
  for every $\varepsilon>0$.
  Now, 
\begin{align*}
m_X(\Lambda \in X: \Lambda \cap M_{s, \mu}^{-\varepsilon} \neq \{0\})
    \xrightarrow{\varepsilon\to 0}
    m_X(\Lambda \in X: \Lambda \cap M_{s, \mu}^{\circ} \neq \{0\}),
\end{align*}  
where $M_{s, \mu}^{\circ}$ is the interior of $M_{s, \mu}$.
  Since measure of the boundary $\partial M_{s, \mu}$ is zero, Siegel's formula \cite{Siegel} implies that $\Lambda\cap \partial M_{s, \mu} = \{0\}$ for almost every $\Lambda\in X$.
  Hence
  $m_X(\Lambda \in X: \Lambda \cap M_{s, \mu}^{\circ} \neq \{0\}) = f(s, \mu)$. The desired claim follows.
\end{proof}
\begin{proof}[Proof of Claim~\ref{claim: algebraic approximation is dynamics}]
  Let $(b_i)_{i=0}^d\in \ZZ^{d+1}$. Let
  \begin{align}\label{eq: b_i c_i relation}
    (\eta,(c_i)_{i=1}^d)
    =a_t\varphi(s)(b_i)_{i=0}^d \in \Lambda_{s, t}.
  \end{align}
  Let $P(x) = \sum_{i=0}^d b_i x^i$, and $s'$ be a root of $P(s)$.
  In particular, $\eta = \euler^{t}P(s)$ and $c_i = \euler^{-t/d}b_i$ so that
  \begin{align*}
    \sum_{i=1}^dc_is^i = \euler^{-t/d}(P(s)-b_0),\quad
    \sum_{i=1}^d i c_is^{i-1} = \euler^{-t/d}P'(s).
  \end{align*}
  By Taylor's Theorem (see \cite[\S 20.3]{kline1998calculus}),
  \begin{align}\label{eq: taylor}
    P(s') = P(s) + (s'-s)P'(s) + \frac{1}{2}P''(s'')(s'-s)^2
  \end{align}
  for some $s''\in [s, s']$.
  Assume we are in Case \ref{case:1} of the claim.
  We get that $P''(s'') = O(\euler^{t/d})$, and hence
  $P(s') = 0 = P(s) + (s'-s)P'(s) + O(\euler^{t/d}(s'-s)^2)$. This implies that
  \begin{align*}
    \left| \sum_{i=1}^di c_is^{i-1}\right|\euler^{-t}\mu = |P'(s)|\euler^{-t(1+1/d)}\mu > |P'(s)||s'-s| & = \left|P(s) + O(\euler^{-t(2+1/d)})\right|             \\
                                                                                                        & = \left|\euler^{-t}\eta + O(\euler^{-t(2+1/d)})\right|,
  \end{align*}
  that is, $|\sum_{i=1}^di c_is^{i-1}|\mu > \left|\eta + O(\euler^{-t(1+2/d)})\right|$.
  Consequently, $(\eta, (c_i)_{i=1}^d)$ satisfies the second condition of $M_{s, \mu}$ up to an $o_t(1)$ error. In addition, we get that $\eta = O(1)$.
  For the first condition, we have that $|c_i|<1$ because $|b_i|<\euler^{t/d}$.
  Now $\sum_{i=1}^d c_is^i = \euler^{-t/d} (P(s) - b_0) = \euler^{-t(1+1/d)} \eta - b_0\euler^{-t/d}$.
  Since $|b_0|<\euler^{t/d}$, we deduce that $\left|\sum_{i=1}^d c_is^i\right| < 1+O(\euler^{-t(1+1/d)})$, and hence $(c_i)_{i=1}^d\in C_s$ up to an $o_t(1)$ error.
  This implies that $(\eta, (c_i)_{i=1}^d)\in M_{s, \mu}^{\varepsilon}$ for all $t$ sufficiently large.

  Assume now the assumptions of Case \ref{case:2}.
  Then $|b_i| = \euler^{t/d}|c_i| <\euler^{t/d}$ for every $i=1, \ldots, d$, by Eq. \eqref{eq: b_i c_i relation}.
  As for $b_0 = \euler^{-t} \eta - \euler^{t/d}\sum_{i=1}^ds^i c_i$, from the definition of $M_{s, \mu}^{-\varepsilon}$ we get that $\left|\sum_{i=1}^ds^i c_i\right| < 1-\delta(s)\varepsilon$, where $\delta(s) = (s^2 + s^4 + \dots + s^{2d})^{-1/2}$.
  Also, $\eta = O_{s,\mu}(1)$ follows.
  This implies that $|b_0| < \euler^{t/d}$ for all $t$ large enough.

  Now it remains to find a root $s'$ of $P(x)$ which is close to $s$.
  For all $s''$ within distance $1$ of $s$ we have $P''(s'') = O(\euler^{t/d})$, and hence \eqref{eq: taylor} implies that for all $s' \in [s-1, s+1]$
  \begin{align}\label{eq: taylor2}
    P(s') = P(s) + (s'-s)P'(s) + O(\euler^{t/d}(s'-s)^2)
  \end{align}
  Note that by the definition of $M_{s, \mu}$, if $\sum_{i=1}^d ic_i s^{i-1}=0$ then also $\eta=0$, and hence $M_{s, \mu}\cap \{(\eta, (c_i)_{i=1}^d): \sum_{i=1}^d ic_i s^{i-1}=0\}$ has trivial interior in $\{(\eta, (c_i)_{i=1}^d): \sum_{i=1}^d ic_i s^{i-1}=0\}$.
  This implies that
  \[\left|\sum_{i=1}^d ic_i s^{i-1}\right| \ge \delta(s)\varepsilon, \]
  and hence $|P'(s)| >\delta(s)\varepsilon \euler^{t/d}$.
  Also, recall that $\eta = O(1)$ so that $P(s) = O(\euler^{-t})$.
  %
  Let $s_0$ be chosen with $s_0 - s = 2\frac{P(s)}{P'(s)}$ so that in particular $|s_0-s| = O(\euler^{-(1+1/d)t})$.
  Then by \eqref{eq: taylor2}
  \begin{align*}
    P(s_0)
     & = -P(s) + O\left(\euler^{t/d} \left( \frac{P(s)}{P'(s)} \right)^2\right)
    \\
     & = -P(s) + O\left(P(s) \euler^{-(1+1/d)t}\right).
  \end{align*}
  Consequently, $P(s_0)$ and $P(s)$ have opposite signs (or they both vanish) and hence there is $s'$ between $s$ and $s_0$ with $P(s') = 0$.
  Moreover, by \eqref{eq: taylor2},
  this $s'$ satisfy that
  \begin{align*}
    -\frac{P(s)}{P'(s)}
     & = (s'-s) + O(\euler^{t/d}(s'-s)^2) \\
     & = (s'-s) + O(\euler^{-(2+1/d)t})
  \end{align*}
  and hence
  \begin{align*}
    s'-s = -\euler^{-(1+1/d)t}\frac{\eta}{\sum_{i=1}^d ic_is^{i-1}}+O(\euler^{-2(1+1/d)t}).
  \end{align*}
  By assumption,  $(\eta\pm\delta, (c_i)_{i=1}^d) \in M_{s, \mu}$ for any $\delta < \varepsilon$ which can be seen to imply $|s'-s| < \mu \euler^{-(1+1/d)t}$.
\end{proof}

\subsection{Proof of Theorem~\ref{thm: bestapprox}}
The result follows directly from the methods of \cite{ShapiraWeiss} and Theorem~\ref{thm: main}, here we merely point out the specific ingredients we use and provide an overview.
Write $X = X_{d+1} = \SL_{d+1}(\R)/\SL_{d+1}(\Z)$, $a_t = \diag(\euler^{t}, \ldots, \euler^{t}, \euler^{-dt})$ for $t \in \R$ and for $\xi \in \R^d$
\begin{align*}
  u(\xi) = \begin{pmatrix}
             1 &   &        & \xi_1  \\
               & 1 &        & \vdots \\
               &   & \ddots & \xi_d  \\
               &   &        & 1
           \end{pmatrix}\in \SL_{d+1}(\R).
\end{align*}
We now recall the construction of a section for the diagonal flow $a_t$ from \cite{ShapiraWeiss} (following mostly their notation).
For $v \in \R^{d+1}$ we let $\bar{v}= (v_1, \ldots, v_d)^t$.
For $r>0$ define a cylindrical set and its top by
\begin{align*}
  C_r & = \{v \in \R^{d+1}: \norm{\bar{v}}\leq r, |v_{d+1}|\leq 1\}, \\
  D_r & = \{v \in \R^{d+1}: \norm{\bar{v}}\leq r, v_{d+1}=1\}.
\end{align*}
Now let $\mathcal{S}_r\subset X$ be the set of lattices with at least one primitive vector in $D_r$ and $\mathcal{S}_r^{\#}\subset \mathcal{S}_r$ the set of lattices with exactly one primitive vector in $D_r$. Choose $r_0>0$ sufficiently large such that $C_{r_0}$ has Lebesgue measure at least $2^{d}$.
By \cite[Lemma 8.4]{ShapiraWeiss}, $\mathcal{S}_{r_0}$ is a cross section in the sense of \cite[Def.~4.2]{ShapiraWeiss} (with respect to the Haar measure $m_X$ on $X$).
By work of Ambrose and Kakutani this induces a finite measure $\mu_{\mathcal{S}_{r_0}}$ on $\mathcal{S}_{r_0}$ with natural properties \cite[Thm.~4.4]{ShapiraWeiss} and with $\mu_{\mathcal{S}_{r_0}}(\mathcal{S}_{r_0}\setminus \mathcal{S}_{r_0}^{\#})=0$.
Moreover, the cross-section $\mathcal{S}_{r_0}$ is $m_X$-reasonable (see \cite[Def.~5.3, Thm.~8.6]{ShapiraWeiss}) which allows us to relate notions of Birkhoff genericity (for $a_t$ and for returns to the cross-section).

The significance of the cross section $\mathcal{S}_{r_0}$ and its subset $\mathcal{S}_{r_0}^{\#}$ can be seen as follows:
let $\xi \in \R^d$ and suppose that $t >0$ is such that $a_tu(\xi)\Z^{d+1}\in \mathcal{S}_{r_0}^{\#}$.
Then there exists a unique $(p, q) \in \Z^{d+1}$ primitive such that $a_t u(\xi)(p, q)^t \in D_{r_0}$, that is, $q = \euler^{dt}$ and $\norm{q^{\frac{1}{d}}(p-q\xi)}\leq r_0$. It is however not guaranteed that the so obtained vector $(p, q)$ is a best approximation.
To that end, for any $\Lambda \in \mathcal{S}_{r_0}^{\#}$ write $v(\Lambda) = (v_{\Lambda},1)$ for the unique primitive vector in $D_{r_0}\cap \Lambda$ and set $r(\Lambda) = \norm{v_\Lambda}$.
Let $\mathcal{B}\subset \mathcal{S}_{r_0}^{\#}$ be the set of lattices $\Lambda$ with
\begin{align*}
  \{\lambda \in \Lambda \cap C_{r(\Lambda)}: \lambda \text{ primitive}\} = \{\pm v(\Lambda)\}.
\end{align*}
Now whenever $a_tu(\xi)\Z^{d+1}\in \mathcal{B}$ it is easy to see that the so obtained vector $(p, q)$ as above is a best approximation.
In fact, all but finitely many best approximations arise in this manner \cite[Prop.~10.4]{ShapiraWeiss}.

We also remark that the set $\mathcal{B}$ is a well-behaved subset of $\mathcal{S}_{r_0}$: it is a positive measure Jordan measurable subset with respect to $\mu_{\mathcal{S}_{r_0}}$ and it is tempered \cite[Lemma 9.2, Prop.~9.8]{ShapiraWeiss}.
This implies that whenever $x \in X$ is an $a_t$-generic point (with respect to $m_X$) the sequence of visits $\{a_tx: a_tx \in \mathcal{B}\}$ is equidistributed with respect to $\nu := \frac{1}{\mu_{\mathcal{S}_{r_0}}(\mathcal{B})}\mu_{\mathcal{S}_{r_0}}|_{\mathcal{B}}$ -- see \cite[Thm.~5.11]{ShapiraWeiss}.
By Theorem~\ref{thm: main} we may apply this to the point $u(\phi(s))\Z^{d+1}$ for almost every $s\in [0,1]$.

It remains to deduce Theorem~\ref{thm: bestapprox} from the above equidistribution result.
The map (see \cite[\S8.2]{ShapiraWeiss})
\begin{align*}
  \Lambda \in  \mathcal{S}_{r_0}^{\#} \mapsto (\pi^{e_{d+1}}(u(-v_{\Lambda})\Lambda), v_{\Lambda}) \in X_d \times \overline{B_{r_0}}
\end{align*}
is continuous and surjective. Moreover, it is easy to check that it associates to any $\xi\in \R^{d}$ and $t>0$ with $a_tu(\xi)\Z^{d+1}\in \mathcal{B}$ the projection lattice and the displacement vector of the corresponding best approximation.
The measure $\mu_{d, \norm{\cdot}}$ in Theorem~\ref{thm: ShapiraWeiss} is the pushforward of $\nu$ under the above map and hence the theorem follows.
Note that the properties of the measure $\nu$ are studied in \cite[\S11]{ShapiraWeiss}.


\providecommand{\bysame}{\leavevmode\hbox to3em{\hrulefill}\thinspace}
\providecommand{\MR}{\relax\ifhmode\unskip\space\fi MR }
\providecommand{\MRhref}[2]{%
  \href{http://www.ams.org/mathscinet-getitem?mr=#1}{#2}
}
\providecommand{\href}[2]{#2}

\end{document}